\documentclass[psamsfonts]{amsart}
\usepackage[utf8]{inputenc}
\usepackage{amsfonts}
\usepackage{hyperref}
\usepackage{comment}
\usepackage{bm}
\usepackage{tikz-cd}
\usepackage{adjustbox} 
\usepackage{amsmath}
\usepackage{xcolor}
\usepackage{amsthm}
\usepackage{bm}
\usepackage{amsfonts}
\usepackage{amssymb}
\usepackage[new]{old-arrows}
\usepackage{tikz-cd}

\hypersetup{
pdftitle={R is DVR with 3 Jordan Holder Factors},
pdfauthor={Xiaoyu Huang},
}

\usepackage[
backend=biber,
style=alphabetic,
sorting=nyt
]{biblatex}
\newtheorem{thm}{Theorem}[section]
\newtheorem{cor}[thm]{Corollary}
\newtheorem{prop}[thm]{Proposition}
\newtheorem{lem}[thm]{Lemma}

\newtheorem{assumption}[thm]{Assumption}
\theoremstyle{definition}
\newtheorem{defn}[thm]{Definition}

\theoremstyle{remark}
\newtheorem{rmk}[thm]{Remark}

\newcommand{\R}{\mathcal{R}}
\newcommand{\Z}{\mathbf{Z}}

\newcommand{\Q}{\mathbf{Q}}

\newcommand{\A}{\mathcal{A}}
\newcommand{\C}{\mathcal{C}}
\newcommand{\cP}{\mathcal{P}}
\newcommand{\E}{\mathcal{E}}

\newcommand{\F}{\mathbf{F}_{p}}

\newcommand{\m}{\mathfrak m}

\newcommand{\mO}{\mathcal{O}}

\newcommand{\D}{{\Bar{D}}}
\newcommand{\GL}{\text{GL}}

\newcommand{\End}{{\text{End}}}
\newcommand{\Hom}{{\text{Hom}}}

\newcommand{\Ext}{{\text{Ext}}}
\newcommand{\G}{{\text{G}_{\Sigma}}}
\newcommand{\red}{{\text{red}}}

\newcommand{\FL}{{\text{FL}}}
\newcommand{\Gal}{{\text{Gal}}}

\newcommand{\tr}{{\text{tr }}}
\newcommand{\mmod}{{\text{mod }}}
\newcommand{\Ann}{{\text{Ann}}}
\newcommand{\tot}{{\text{tot}}}
\newcommand{\sss}{{\text{ss}}}
\newcommand{\Frob}{{\text{Frob}}}
\newcommand{\ur}{{\text{ur}}}
\newcommand{\f}{{\textbf{f}}}
\newcommand{\adele}{\mathbf{A}}

\newcommand{\upperhalf}{\mathbf{H}}
\newcommand{\ad}{\text{ad}}

\newcommand{\Rpseudo}{R_{\D}^{\FL}}
\newcommand{\tpsi}{\Tilde{\psi}}
\newcommand{\val}{\text{val}}

\usepackage[OT2,T1]{fontenc}
\DeclareSymbolFont{cyrletters}{OT2}{wncyr}{m}{n}
\DeclareMathSymbol{\Sha}{\mathalpha}{cyrletters}{"58}

\makeatletter
\let\c@equation\c@thm
\makeatother
\numberwithin{equation}{section}

\title[\resizebox{4.5 in}{!}{Deformation rings with three Jordan-H\"{o}lder factors}]{On deformation rings of residual Galois representations with three Jordan-H\"{o}lder factors And Modularity}

\author{Xiaoyu Huang}

\makeatletter
\def\keywords{\xdef\@thefnmark{}\@footnotetext}
\makeatother

\begin{document}
\keywords{2010 \emph{Mathematics Subject Classification.} Primary 11F80; Secondary 11F55.}
\keywords{\emph{Keywords.} Galois representations, automorphic forms, deformations, the ideal of reducibility.}

\begin{abstract}
In this paper, we study Fontaine--Laffaille, self-dual deformations of a
mod $p$ non-semisimple Galois representation of dimension $n$ with its Jordan-H\"{o}lder factors being
three mutually non-isomorphic absolutely irreducible representations. We show
that under some conditions regarding the orders of certain Selmer groups, the
universal deformation ring is a discrete valuation ring. Given enough information on the Hecke algebra, we also prove an $R = \mathbf{T}$ theorem in the general context. We then apply our results to abelian surfaces with cyclic rational isogenies and certain $6$-dimensional representations arising from automorphic forms congruent to Ikeda lifts. Assuming the Bloch--Kato conjecture, our result identifies special $L$-value conditions for the existence of a unique abelian surface isogeny class and an $R = \mathbf{T}$ theorem for certain $6$-dimensional Galois representations.
\end{abstract}

\maketitle


\tableofcontents
\addtocontents{toc}{\protect\setcounter{tocdepth}{1}}

\section{Introduction}
\subsection{Context}
In 1997, Skinner and Wiles proved the first modularity theorem for ordinary residually reducible 2-dimensional Galois representations over $\Q$ \cite{skinner1997ordinary}. Their method, which was an extension of the Taylor-Wiles strategy, had a limitation that is sometimes referred to as the ``cyclicity'' hypothesis. This hypothesis essentially guaranteed that the underlying non-semi-simple residual Galois representation with specified characters on the diagonal in their case was unique up to isomorphism. In a subsequent paper, \cite{skinner1999residually}, they removed this obstruction using a quite different approach. In particular, they no longer identified a universal deformation ring with a Hecke algebra.

In the past 15 years, there have been various extensions of both approaches. An analogous method to that of \cite{skinner1999residually} was used for example by Thorne \cite{thorne2015automorphy} and later by Allen, Newton and Thorne \cite{allen2020automorphy} to prove an automorphy lifting theorem in the context of $\GL_n$.

On the other hand, Berger and Klosin used the Taylor-Wiles strategy and various versions of the cyclicity hypothesis to prove modularity of certain 2-dimensional Galois representations over imaginary quadratic fields (cf. \cite{berger2009deformation} \cite{berger2011r}). Later they considered a more general situation where the residual representation was a non-split extension of 2 irreducible Galois representations of arbitrary dimensions. This led to a proof of modularity of certain 4-dimensional Galois representations which were shown to arise from type-G Siegel modular forms congruent to Yoshida lifts \cite{berger2013deformation}. In a somewhat different direction, Wake and Wang-Erickson \cite{wake2020rank} proved an $R = \mathbf{T}$ theorem for 2-dimensional Galois pseudorepresentations of $\Q$ arising in the context of the Mazur's Eisenstein ideal using the method of pseudodeformations, which was introduced in detail in their companion paper \cite{wake2019deformation}. The cyclicity hypothesis was no longer required for this method, though it was needed to establish the $R = \mathbf{T}$ theorem. In their later result on the square-free level Eisenstein ideal \cite{wake2021eisenstein}, the cyclicity hypothesis was no longer required in the proof of the $R = \mathbf{T}$ theorem. 

A common feature of the results of Berger-Klosin and Wake-Wang-Erickson is the fact that the residual Galois representation under consideration has two Jordan-H\"{o}lder factors. In \cite{berger2013deformation} the authors pinpointed conditions under which one can control a universal deformation ring of a reducible Galois representation with two Jordan-H\"{o}lder factors, and in favorable situations proved an $R = \mathbf{T}$ theorem. 

\subsection{Main Results}
In this paper, we carry out a similar analysis in the case where the residual representation has three Jordan-H\"{o}lder factors. In particular, we find an appropriate replacement for the cyclicity hypothesis and prove the uniqueness of the residual representation under such replacement. We then show that under certain Selmer group conditions, the universal deformation ring is a discrete valuation ring (DVR), which is the easiest case when modularity can sometimes be proved. To the best of our knowledge, Calegari first introduced a proof of $R = \mathbf{T}$ theorems in the setting when $R$ is a DVR \cite{calegari2006eisenstein}. Finally, we prove an $R = \mathbf{T}$ theorem in a more general context when the universal deformation ring $R$ is not necessarily a DVR. The main theorem is stated below, and more details can be found in Theorem \ref{main thoerem} and Theorem \ref{R=T}. Please see Section \ref{Global Selmer Groups} for the definition of $H^{1}_{\Sigma}$. 

Let $n$ be an integer and $p > n$ be a prime number. Let $\mO$ be the valuation ring of a finite extension of $\Q_p$ with residue field $k$ and uniformizer $\varpi$. Let $F$ be a number field and $G_{\Sigma} = \Gal(F_{\Sigma}/F)$. The residual Galois representation $\overline{\sigma}: \G \to \GL_{n}(k)$ satisfies
\begin{equation*}
\overline{\sigma} = 
\begin{pmatrix}
    \rho_1 & a & b \\ 
    0 & \rho_2 & c \\ 
    0 & 0 & \rho_3 \\ 
\end{pmatrix},
\end{equation*}
where 
\begin{enumerate}
\item $\overline{\sigma}$ is Fontaine-Laffaille at $p$;
\item $\rho_i: G_{\Sigma} \to \GL_{n_i}(k)$ are absolutely irreducible for each $i$ and $\rho_i|_{G_{F_{v}}} \not \cong \rho_j|_{G_{F_{v}}}$ for $i \ne j$ and $v|p$;
\item there is a unique deformation $\Tilde{\rho}_i$ of $\rho_i$ to $\mO$ that is Fontaine-Laffaille at $p$ up to strict equivalence for all $i$;
\item $\begin{pmatrix}
\rho_1 & a \\
0 & \rho_2
\end{pmatrix}
$
and 
$
\begin{pmatrix}
    \rho_2 & c \\
    0 & \rho_3
\end{pmatrix}
$
are non-split extensions;
\item $\overline{\sigma}$ is $\tau$-self-dual for some anti-involution $\tau: k[\G] \to k[\G]$.
\end{enumerate}

\begin{thm}{\label{intro main}}
Consider the deformations $\sigma: G_{\Sigma} \to \GL_{n}(\overline{\Q_{p}})$ of $\overline{\sigma}$ with the following deformation conditions:
\begin{enumerate}
    \item $\sigma$ is Fontaine-Laffaile at $p$.
    \item $\sigma$ is $\tau$-self-dual.
\end{enumerate}
Let $R$ be the reduced quotient of the related universal deformation ring. Suppose
\begin{enumerate}
\renewcommand{\labelenumi}{(\roman{enumi})}
    \item $H^{1}_{\Sigma}(F, \Hom(\rho_3,\rho_1)) = 0$;
    \item $\dim H^{1}_{\Sigma}(F, \Hom(\rho_2,\rho_1)) = \dim H^{1}_{\Sigma}(F, \Hom(\rho_1,\rho_2))= 1$;
    \item $\# H^1_{\Sigma}(F,\Hom(\Tilde{\rho_3},\Tilde{\rho_2}) \otimes_{E} E/\mO) \le \# \mO/\varpi^{r} \mO$.
\end{enumerate}
If $r = 1$ and $\overline{\sigma}$ admits a deformation to $\GL_n(\mO)$ that is Fontaine-Laffaille and $\tau$-self-dual, then $R$ is a DVR. 

If $r \ge 1$ and $I^{\tot}$ is the total reducibility ideal of $R$ (Definition \ref{red def}), $\mathbf{T}$ is a suitable Hecke algebra, and $\# \mathbf{T}/ \phi(I^{\tot}) \ge \# \mO/\varpi^r$ for a cannonical map $\phi: R \to \mathbf{T}$, then $\phi$ is an isomorphism.
\end{thm}

Theorem \ref{intro main} is intended for a broad application and we believe the assumptions should be met in many cases. In particular, when $H^1_{\Sigma}(\cdot)$ coincides with the Bloch-Kato Selmer group $H^1_f(\cdot)$, in our experience, conditions (i)-(iii) are met often (see Lemma \ref{two selmer groups} when the two groups agree and Remark \ref{comments on Selmer group sizes} for examples when the vanishing condition (i) or cyclicity condition (ii) holds).



Needless to say, there are significant new challenges that arise when dealing with three Jordan-H\"{o}lder factors as opposed to two. Let us briefly remark on some of them here, as the novelty of the paper is in part in overcoming some of them. The first issue, as stated above, is to identify an appropriate replacement for the cyclicity hypothesis and prove the uniqueness of the residual representation under that replacement. The second issue is the existence of a Galois-stable lattice with respect to which the residual representation is indeed the reduction of a desired irreducible representation arising from geometry or automorphic forms. Specifically, when considering two Jordan-H\"{o}lder factors, one could invoke Ribet's Lemma \cite[Proposition 2.1]{ribet1976modular} to guarantee the existence of such lattices for non-trivial residual extensions in ``either direction''. However, when there are three Jordan-H\"{o}lder factors, Ribet's Lemma no longer applies. Moreover, a theorem of Bella\"{i}che known as ``generalized Ribet's Lemma'' does not guarantee the existence of such lattices (see Remark \ref{Ribets lemma}). Here, however, using the condition that our representations are Fontaine--Laffaille at $p$, we are able to construct a lattice such that the residual representation is composed of non-trivial extensions, giving an analog of Ribet's lemma. We note that our argument also applies to any condition that is preserved when taking subquotients.

Lastly, when studying the structure of the universal deformation ring $R$, one can study the reducibility ideal of $R$, details of which can be found in Section 1.5 \cite{bellaiche2009families}. The reducibility ideal has an explicit expression using the theory of GMAs. For representations with two Jordan-H\"{o}lder factors, such expression is simple and one can compute the ideal using it. For $n$ Jordan-H\"{o}lder factors, however, the expression is a sum of $\binom{n}{2}$ $R$-modules, which significantly complicates the calculations. In \cite{berger2020deformations}, the authors proved the principality of this ideal in a special case of deformations with three particular Jordan-H\"{o}lder factors, using explicit injection maps from certain Hom groups to Ext groups of $\G$-modules and the theory of pseudocharacters. 

Here, we study the reducibility ideal in the general context of pseudorepresentation developed by \cite{bellaiche2009families}, \cite{chenevier2014p} and \cite{wake2019deformation}. In particular, in \cite{wake2019deformation}, the authors studied deformation conditions for pseudorepresentations, extending the result of \cite{ramakrishna1993variation} from representations to pseudorepresentations, which allows us to define a Fontaine-Laffaille at $p$ deformation ring for pseudorepresentations. They also showed the injections from \cite{berger2020deformations} are in fact isomorphisms in the context of universal deformations for pseudorepresentation of degree $2$. We prove such isomorphisms exist in the case of universal deformations for pseudorepresentations of any degree in the 3 Jordan-H\"older factors setting, and conclude the principality of the reducibility ideal under certain Selmer group assumptions. We expect a similar proof should hold for an arbitrary number of Jordan-H\"older factors. We believe our results in Section 4 can also be useful when one removes the cyclicity hypothesis and study deformations of pseudorepresentations in the $n \ge 3$ Jordan-H\"older factors setting, similar to the work of \cite{wake2020rank}\cite{wake2021eisenstein}.

Residual Galois representations with three Jordan-H\"{o}lder factors arise naturally in some important number-theoretic questions. We present a few examples where these occur, without delving deeply into the specific details or contributions of these papers. For example, Hsieh \cite{Hsieh11} \cite{hsieh2014eisenstein} studied generically irreducible and residually reducible Galois representations with three Jordan-H\"{o}lder factors associated with cuspidal automorphic representations of $U(3, 1)$ and $U(2, 1)$ in the context of the Iwasawa main conjectures for $\GL_2 \times \mathcal{K}^{\times}$ and CM fields respectively using the method of Eisenstein congruences, where $\mathcal{K}$ is an imaginary quadratic field. Bella\"{i}che and Chenevier
\cite{bellaiche2004formes} \cite{bellaiche2009families} studied residually reducible Galois representations with three Jordan-H\"{o}lder factors in the context of the Bloch--Kato conjecture.

In particular, one of the most interesting open problems in number theory is the Paramodular Conjecture \cite{brumer2014paramodular} postulating that there is a one-to-one correspondence between isogeny classes of abelian surfaces $A_{/ \Q}$ of conductor $N$ with $\text{End}_{\Q} A = \Z$ or QM abelian surfaces $A_{/ \Q}$ of conductor $N^2$ and cuspidal, nonlift weight 2 Siegel paramodular newforms $f$ with rational eigenvalues, up to scalar multiplication.

Let $p > 3$ be a prime. If the abelian surface $A$ in question has a rational $p$-isogeny and a polarization of degree prime to $p$, then the Galois representation on $A[p]$, $\overline{\sigma}_A: G_{\Q} \to \text{GL}_{4}(\F)$ satisfies
\begin{equation*}
   \overline{\sigma}_A^{\sss} \cong \psi \oplus \rho^{\sss} \oplus \psi^{-1} \chi,
\end{equation*}
where $\chi$ is the mod-$p$ cyclotomic character, $\psi: G_{\Q} \to \F^{\times}$ is the character acting on the Galois-stable cyclic subgroup of order $p$ defined over $\Q$, and $\rho$ is a two-dimensional representation. If $\rho$ is irreducible then this is precisely the situation that we study in this paper. In the special case when $A$ had a $p$-torsion point, so $\psi$ is trivial, Berger and Klosin \cite{berger2020deformations} formulated conditions under which the universal deformation ring is a DVR. They assumed the conductor $N$ of $A$ is square-free and the local Galois representation at primes dividing the conductor has a particular shape 
 (the monodromy rank one condition). 

In this paper, we show the deformation ring is a DVR by using certain Selmer group conditions instead, and we do not require $N$ to be square-free. In fact, if $\psi$ is non-trivial, the conductor $N$ is always square-full. The Selmer groups in our context are the ``relaxed'' Selmer groups $H^1_{\Sigma}$ (see Section \ref{Global Selmer Groups}), which only impose local conditions at $p$. In \cite{berger2020deformations}, the authors showed that the orders of $H^1_{\Sigma}$ can be identified with the orders of Bloch--Kato Selmer groups $H^1_{f}$ when $N$ is square-free, and the latter are conjecturally are bounded by special $L$-values. Here, we prove the relationship without the square-free assumption. For a more detailed discussion, please see Section \ref{section $R$ is a DVR}. Let $\epsilon$ be the $p$-adic cyclotomic character. The following summarizes Theorem \ref{R is DVR}.

\begin{thm}
Let $p > 3$ and $A$ be an abelian surface with a rational $p$-isogeny, a polarization degree prime to $p$ where $\overline{\sigma}_A^{\sss} \cong \psi \oplus \rho \oplus \psi^{-1} \chi$ where $\rho$ is absolutely irreducible. Let $\Tilde{\psi}$ be the Teichm\"uller lift of $\psi$. Assume there is a unique Fontaine-Laffaille deformation $\Tilde{\rho}$ of $\rho$ up to strict equivalence class, and 
\begin{enumerate}
    \item $H^{1}_{\Sigma}(\Q,\psi^{2}\chi^{-1}) = 0$;
    \item $\dim_{\F} H^{1}_{\Sigma}(\Q,\overline{\rho}_f(\psi^{-1}))= 1$;
    \item $\# H^1_{\Sigma}(\Q,\Tilde{\rho}(\tpsi \epsilon^{-1}) \otimes E/\mO) \le p$.
\end{enumerate}
Then $R$ in Theorem \ref{intro main} is a DVR and there is a unique isogeny class of $A_{/\Q}$ of conductor $N$ with a rational $p$-isogeny and $\overline{\sigma}_A^{\sss} \cong \psi \oplus \rho \oplus \psi^{-1} \chi$.
\end{thm}

Another context where such representations arise is that of automorphic forms congruent and orthogonal to Ikeda lifts. Let $\phi$ be a classical $\GL_2$-newform of weight $2k'$ and level $1$, $\rho_{\phi}: G_{\Q} \to \GL_2(\overline{\Q_{p}})$ be the $p$-adic representation associated to $\phi$ and $\rho_{\phi}(i)$ be its $i$th Tate twist. Let $K$ be an imaginary quadratic extension of $\Q$. Ikeda \cite{ikeda2001lifting} showed that there exists a Hecke-equivariant lift $I_{\phi}$ such that the associated Galois representation  $\rho_{I_{\phi}}: G_K \to \GL_{6}(\overline{\Q}_p)$ is
$
\rho_{I_{\phi}} \cong \epsilon^{2-k'} \otimes \bigoplus_{i = 0}^{2} \rho_{\phi}(i)|_{G_{K}},
$
where $\epsilon$ is the $p$-adic cyclotomic character. In \cite{brown2020congruence}, Brown and Klosin studied the sufficient conditions when there is an automorphic form $f'$ on $U(3,3)$ that is orthogonal to the space spanned by all the Ikeda lifts such that $f'$ is congruent to $I_{\phi}$.

If such $f'$ exits, the residual Galois representation attached to $f'$, $\overline{\rho}_{f'}: G_{K} \to \GL_6(\overline{\F})$ also satisfies
$
\overline{\rho}_{f'}^{\sss} \cong  \overline{\epsilon}^{2-k'} \otimes \bigoplus_{i = 0}^{2} \overline{\rho}_{\phi}(i)|_{G_{K}}.
$
Let $\sigma$ be the Galois representation attached to $f'$. Assuming that $\sigma$ is absolutely irreducible and there is no newform $h \in S_{2k'}(1)$ such that $\phi \equiv h (\mmod \varpi)$, we have the following theorem stating the $\lambda$-part of the Bloch-Kato conjecture implies modularity of $\sigma$. For more details, see Theorem \ref{Ikeda R=T}.

\begin{thm}
Let $\sigma: G_{K} \to \GL_6(E)$ be absolutely irreducible, Fontaine-Laffaille at $p$ and unramified away from $p$. Suppose $\overline{\sigma}^{\sss}$ $\cong \overline{\epsilon}^{2-k'} \otimes \bigoplus_{i = 0}^{2} \overline{\rho}_{\phi}(i)|_{G_{K}}$.
If $v_1, v_2$ and $L^{\text{alg}}(\cdot,0)$ are as defined in Section \ref{Ikeda lifts} and $\chi_K$ is the quadratic character associated with $K$, and the assumptions in Theorem \ref{Ikeda R=T} are satisfied.

Then if
\begin{enumerate}
    \item $v_1 \ge 1$,
    \item $v_2 = 0$,
    \item $\text{val}_{\varpi}(\# \mO/L^{\text{alg}}(\text{ad}^0 \rho_{\phi},0)) + \text{val}_{\varpi}(\# \mO/L^{\text{alg}}(\text{ad}^0 \rho_{\phi} \otimes \chi_K,0)) = 1$,
    \item $\text{val}_{\varpi}(\# \mO/L^{\text{alg}}(\text{ad}^0 \rho_{\phi}(2),0)) = 0$,
    \item $\text{val}_{\varpi}(\# \mO/L^{\text{alg}}(\text{ad}^0 \rho_{\phi}(3)  \otimes \chi_K,0))  = 0$,
\end{enumerate}
then $\lambda$-part of the Bloch-Kato conjecture implies $\sigma$ is modular.
\end{thm}

\subsection{Methods}
Let us explain the strategy in our proof in more detail. Let $F$ be a number field. Let $E$ be a finite extension of $\Q_p$, $\mO$ be its valuation ring with residue field $k$. Let $\Sigma$ be a finite set of finite places of $F$ which contains $\Sigma_p$, the set of all places above $p$. Let $F_{\Sigma}$ be the maximal extension of $F$ unramified outside of $\Sigma$. Let $G_{\Sigma} = \Gal(F_{\Sigma}/F)$. In this paper, we study the general situation of a residual Galois representation $\overline{\sigma}: \G \to \GL_{n}(k)$ of the form:
$$
\overline{\sigma} = 
\begin{pmatrix}
    \rho_1 & a & b \\ 
    0 & \rho_2 & c \\ 
    0 & 0 & \rho_3 \\ 
\end{pmatrix},
$$
where $\rho_i: \G \to \GL_{n_{i}}(k)$ are absolutely irreducible for each $i$ and $\rho_i|_{G_{F_{v}}} \not \cong \rho_j|_{G_{F_{v}}}$ for $i \ne j$ and $v|p$. We assume it is Fontaine-Laffaile at $p$ and ``$\tau$-self-dual'' for some anti-involution $\tau: \mO[\G] \to \mO[\G]$ which interchanges $\rho_1$ with $\rho_3$ and leaves $\rho_2$ fixed.

Our first result is Proposition \ref{Uniqueness of residual} where we prove uniqueness of $\rho$ under the assumptions that $H^1_{\Sigma}(F, \Hom(\rho_3, \rho_1)) = 0$ and $\dim H^1_{\Sigma}(F, \Hom(\rho_2, \rho_1)) = 
\newline
\dim H^{1}_{\Sigma}(F, \Hom(\rho_3,\rho_2)) = 1$. These Selmer group conditions can be considered a replacement for the ``cyclicity condition'' in \cite{skinner1997ordinary}. 

In the first part of our main result, we consider the reduced universal self-dual Fontaine-Laffaile at $p$ deformation ring $R$ for $\overline{\sigma}$ and prove that under a further assumption that $\dim H^1_{\Sigma}(F, \Hom(\rho_1, \rho_2)) = 1$, its total reducibility ideal $I^{\tot}$ (Definition \ref{red def}) is principal. In fact, we prove this result in a more general context of pseudorepresentations as considered in \cite{wake2019deformation}. Then we control the size of $R/I^{\tot}$, which measures the reducible deformations of $\overline{\sigma}$, by putting restrictions on $\dim H^1_{\Sigma}(F, \Hom(\rho_2, \rho_1))$ and $H^1_{\Sigma}(F,\Hom(\Tilde{\rho_2},\Tilde{\rho_1}) \otimes E/\mO)$. We show that if $\# H^1_{\Sigma}(F,\Hom(\Tilde{\rho_2},\Tilde{\rho_1}) \otimes_{E} E/\mO) \le \# \mO/\varpi \mO$, where $\Tilde{\rho_i}$ is the assumed unique Fontaine--Laffaille deformation of $\rho_i$ to $\GL_{n_{i}}(\mO)$, then $R$ is a DVR and $I^{\tot}$ is its maximal ideal. It follows that there is a unique deformation $\sigma$ of $\overline{\sigma}$ to characteristic zero rings. 

In the second part of our main result, we also consider the case where 
\newline
$\# H^1_{\Sigma}(F,\Hom(\Tilde{\rho_2},\Tilde{\rho_1}) \otimes_{E} E/\mO) \le \# \mO/\varpi^s \mO$ with $s > 1$ when $R$ might not be a DVR. We show that given certain lower bounds on the appropriate Hecke algebra $\mathbf{T}$, we could conclude an $R = \mathbf{T}$ theorem. First, we show that given our new cyclicity condition on Selmer groups, if an absolutely irreducible Galois representation $\sigma': G_{F} \to \GL_{n}(\overline{\Q}_p)$ satisfies $\overline{\sigma^{'}}^{\sss} \cong \rho_1 \oplus \rho_2 \oplus \rho_3$ and is Fontaine--Laffaille at $p$, then there exists a Galois-stable lattice that guarantees the existence of a non-trivial residual extension in the desired direction, resembling Ribet's Lemma. Then we show that if there is a matching bound $\# \mathbf{T}/J \ge \# \mO/\varpi^s \mO$ where $J$ is a relevant congruence ideal, we obtain an $R = \mathbf{T}$ theorem.

In the last sections, we study two examples where our results can be applied. The first example is that of abelian surfaces with rational $p$-isogenies and polarization degrees prime to $p$. We show that given the above Selmer group conditions, the deformation ring $R$ is a DVR. As a Corollary, we show such surfaces are unique up to isogeny. The second example is Galois representations attached to automorphic forms congruent and orthogonal to Ikeda lifts. Applying our $R = \mathbf{T}$ theorem, we show that the Bloch--Kato conjecture implies modularity in this case.

\subsection{Organization}
In Section 2, we introduce the common notations in this paper and introduce Selmer groups and their properties. In Section 3, we define the residual representation setup, its deformation conditions, and conditions under which it is unique. We also introduce the main running assumptions in this paper in this section. In Section 4, we introduce pseudorepresentations and the reducibility ideal, and the conditions under which the reducibility ideal is principal. In Section 5, we study the reducible deformations and prove our main theorem 

\subsection{Acknowledgement}
The author would like to thank Krzysztof Klosin for initially suggesting the problem and for detailed guidance. We also thank the anonymous referee for their helpful comments. We also would like to thank Tobias Berger, Ken Kramer, Carl Wang Erickson, Jaclyn Lang, Preston Wake and Ga\"{e}tan Chenevier for enlightening discussions and comments. Finally, we are grateful to Shiva Chidambaram and Andrew Sutherland for their help with computing examples of suitable abelian surfaces.

\section{Notation and Selmer Groups} 
In this section, we introduce the recurring notations in the paper as well as various types of Selmer groups. We begin by defining Fontaine--Laffaille local finite structures for \( \ell = p \) and unramified local finite structures for primes \( \ell \neq p \). These lead to the definitions and discussions of the Bloch--Kato Selmer groups and the ``relaxed'' Selmer groups in the global context. These Selmer groups will be instrumental in the subsequent discussions of the reducibility and modularity of Galois representations.

\subsection{Notation}
Let $p > 2$ be a prime number. Let $E$ be a finite extension of $\Q_p$, $\mO$ be its valuation ring with residue field $k$ and uniformizer $\varpi$.

Let $F$ be a number field. For each place $v$, we fix the embeddings $\overline{F} \hookrightarrow \overline{F}_{v}$. Let $\Sigma$ be a finite set of finite places of $F$ which contains $\Sigma_p$, the set of all places above $p$. Let $F_{\Sigma}$ be the maximal extension of $F$ unramified outside of $\Sigma$. Let $G_{\Sigma} = \Gal(F_{\Sigma}/F)$.

\subsection{Local Cohomology}{\label{Local Cohomology}}
We call $M$ a $p$-adic $G_{\Sigma}$-module over $\mO$ if $M$ is an $\mO$-module with an $\mO$-linear action of $G_{\Sigma}$ and is one of the following:
 \begin{enumerate}
     \item a finitely generated $\Z_p$-module,
     \item a finite dimensional $\Q_p$-vector space,
     \item a discrete torsion $\Z_p$-module.
 \end{enumerate}
 
In this section, suppose $\ell$ is a prime and $K$ is a finite extension over $\Q_{\ell}$. Let $V$ be a vector space over $E$ with a continuous linear action of $G_K$. Let $T \subseteq V$ be a $G_K$-stable $\mO$-lattice and let $W = V/T$. For $n \ge 1$, define
$$
W_n = \{x \in W: \varpi^n x = 0 \} \cong T/\varpi^n T.
$$

Following the terminology of Bloch--Kato \cite[Section 3]{bloch2007functions} and Rubin \cite[\S 1.3.2]{rubin2014euler}, we define special subgroups $H^1_f(K,M)$ of certain cohomology groups $H^1(K,M)$.
\begin{defn}
A \textbf{finite structure} on $M$ is a choice of an $\mO$-submodule $H^1_f(K,M) \subseteq H^1(K,M)$.
\end{defn}
\begin{rmk}
We use the term ``finite structure'' following \cite{bloch2007functions} and \cite{rubin2014euler} for consistency. It can be understood as a ``finite Selmer structure'' and should not be confused with a finite-flat structure.
\end{rmk}
We will discuss two different local finite structures on $M$, depending on whether $\ell = p$ or $\ell \ne p$.

\subsubsection{\textsf{$\ell = p$ and Fontaine--Laffaille}} 

Suppose $K$ is unramified over $\Q_p$. Following Bloch and Kato \cite{bloch2007functions}, we define the \textbf{Fontaine--Laffaille local finite structure} as follows.

We set 
$$
H^1_f(K,V) = \ker(H^1(K,V) \to H^1(K,B_{\text{crys}} \otimes V))
$$
(see Section 1 in \cite{fontaine1987p} for details of $B_{\text{crys}}$) and denote $H^1_f(K,T)$ and $H^1_f(K,W)$ as the inverse image and the image of $H^1_f(K,V)$ via the natural maps
$$
H^1(K,T) \to H^1(K,V) \to H^1(K,W),
$$
respectively. $V$ is a \textbf{Fontaine--Laffaille at $p$} $G_K$-representation over $\Q_p$ if for all places $v|p$, $\text{Fil}^0 D = D$ and $\text{Fil}^{p-1} D = (0)$ for the filtered vector space $D = (B_{\text{crys}} \otimes_{\Q_p} V)^{G_{F_{v}}}$ defined by Fontaine.

\begin{rmk}
Alternatively, we could follow the theory of Fontaine--Laffaille \cite{Fontaine1982} and Section 2.4.1 in \cite{clozel2008automorphy}. For finitely generated $p$-adic $G_K$-modules, let $\mathcal{MF}_{\mO}$ denote the category of finitely generated $\mO$-modules $M$ together with 
\begin{itemize}
    \item a decreasing filtration $\text{Fil}^i M$ by $\mO$-submodules which are $\mO$-direct summands with $\text{Fil}^0 M = M$ and $\text{Fil}^{p-1} M = (0)$;
    \item Frobenius linear maps $\Phi^i: \text{Fil}^i M \to M$ with $\Phi^i|_{\text{Fil}^{i+1} M} = p \Phi^{i+1}$ and $\sum_{i} \Phi^i \text{Fil}^i M = M$.
\end{itemize}

There is an exact, fully faithful, covariant functor $\mathbf{G}$ from $\mathcal{MF}_{\mO}$ to the category of finitely generated $\mO$-modules with a continuous action by $G_{K}$. Its essential image is closed under taking sub-objects and quotients. The essential image of $\mathbf{G}$ also contains quotients of lattices in Fontaine--Laffaille at $p$ representations.

For any $p$-adic $G_K$-module $M$ of finite cardinality that is in the essential image of $\mathbf{G}$, we define the Fontaine--Laffaille local finite structure $H^1_f(K,M)$ as the image of $\Ext^1_{\mathcal{MF_{O}}}(1_{\text{FD}},D)$ in $H^1(K,M) \cong \Ext^1_{\mathcal{O}[G_K]}(1,M)$, where $M = \mathbf{G}(D)$ and $1_{\text{FD}}$ is the unit filtered Dieudonné module defined in Lemma 4.4 in \cite{bloch2007functions}.

Though two different definitions of $H^1(K,W_n)$ were provided: one through the pull-back of $H^1_f(K,W)$ using the identification $W = \displaystyle{\lim_{\longrightarrow}} W_n$, and the other utilizing the $\mathbf{G}$-functor, these definitions were shown to agree (cf. Proposition 2.2 in \cite{diamond2004tamagawa}).
\end{rmk}

\subsubsection{\textsf{$\ell \ne p $}} 
For primes $\ell \ne p$, we define the \textbf{unramified local-finite structure} on any $p$-adic $G_K$-module $M$ over $\mO$ as 
$$
H^1_{\text{ur}}(K,M) = \ker(H^1(K,M) \to H^1(K_{\text{ur}},M)),
$$
where $K_{\text{ur}}$ is the maximal unramified extension of $K$. 

We again follow Bloch--Kato's definition of the finite structures on $V,T$, and $W$. Let
$$
H^1_f(K,V) = H^1_{\ur}(K,V),
$$
and $H^1_f(K,T)$ and $H^1_f(K,W)$ be the inverse image and the image of $H^1_f(K,V)$ via the natural maps
$$
H^1(K,T) \to H^1(K,V) \to H^1(K,W).
$$
We will call these minimally ramified structures on $V,T,W$ respectively. 




\subsection{Global Selmer Groups}{\label{Global Selmer Groups}}
Assume $p$ is unramified at $F/\Q$. 
Let $G_{\Sigma} = \Gal(F_{\Sigma}/F)$. Let $M$ be a $p$-adic $G_{\Sigma}$-module. 

We define the local finite structures on $M$ as follows:
\begin{itemize}
    \item if $v|p$, $H^1_f(F_v,M)$ is the Fontaine--Laffaille local finite structure;
    \item if $v \nmid p$, $H^1_f(F_v,M)$ is the unramified local-finite structure;
    \item if $v | \infty$, $H^1_f(F_{v}, M) = 0$.
\end{itemize}

The \textbf{Bloch--Kato Selmer group} $H^1_{f}(F,M)$ is defined as 
$$
H^1_{f}(F,M) = \ker(H^1(G_{\Sigma},M) \to \prod_{v \in \Sigma}(H^1(F_{v},M)/H^1_f(F_v,M))).
$$
and the \textbf{``relaxed'' Selmer group} $H^1_{\Sigma}(F,M)$ is defined with no condition imposed at primes $\ell \in \Sigma \setminus \Sigma_{p}$:
$$
H^1_{\Sigma}(F,M) = \ker(H^1(G_{\Sigma},M) \to \prod_{v \in \Sigma_{p}}(H^1(F_{v},M)/H^1_f(F_v,M))).
$$

\begin{lem}{\label{torsion Selmer is Selmer torsion}}
Suppose $W^{\G} = 0$. Then $H^1_{\Sigma}(F,W_n) = H^1_{\Sigma}(F,W)[\varpi^n]$.
\end{lem}
\begin{proof}
Suppose $n$ is nonzero, following the same argument as in \cite{rubin2014euler} Lemma 1.2.2, there is an exact sequence
$$
0 \to W^{\G}/n  W^{\G} \to H^1_{\Sigma}(K,W_n) \to H^1_{\Sigma}(K,W)[\varpi^n] \to 0.
$$
In particular, if $W^{\G} = 0$. Then $H^1_{\Sigma}(F,W_n) = H^1_{\Sigma}(F,W)[\varpi^n]$.
\end{proof}

The two Selmer groups concide in certain situations. For example, if $\Sigma = \Sigma_p$, then $H^1_{\Sigma}(F,M) = H^1_{f}(F,M)$. The following Lemma is from 5.6, 5.7 in \cite{berger2013deformation}. It identifies sufficient conditions where the two Selmer groups coincide.

\begin{lem}{\label{two selmer groups}}
Let $V^* = \Hom_{\mO}(V,E(1))$, $T^* = \Hom_{\mO}(T,\mO(1))$, and $W^* = V^*/T^*$. Define the $v$-Euler factor as
$$
P_{v} (V^*,X) = \det (1 - X \Frob_v|_{(V^*)^{I_{v}}}).
$$
Then $H^1_{\Sigma}(F,W) = H^1_f(F,W)$ if for all places $v \in \Sigma, v \nmid p$ we have 
\begin{enumerate}
    \item $P_{v}(V^*,1) \in \mO^{*}$,
    \item $\text{Tam}^0_v(T^*) = 1$,
\end{enumerate}
where the Tamagawa factor $\text{Tam}^0_v(T^*)$ equals $\# H^1(F_v,T^*)_{\text{tor}} \times |P_{v}(V^*,1)|_p$ (cf. \cite{fontaine1992Lvalues}). Moreover, if $W^{I_{v}}$ is divisible, then $\text{Tam}^0_v(T^*) = 1$.
\end{lem}

\section{Deformation Problem and Main Assumptions}
In this section, we outline the key assumptions and structures for studying deformations of residual Galois representations. We begin by introducing the residual representation \( \sigma \), which has three Jordan–Hölder factors. The conditions under which this residual representation is unique are established, using Fontaine--Laffaille theory and assumptions related to Selmer groups. The running assumptions in the paper are stated in the main assumptions (Assumption \ref{main assumptions}), which include the vanishing and constraints on certain Selmer groups.

\begin{defn}
Let $n \ge 3$. Let $p > n$ be a prime number. Let $\hat{\mathcal{C}}_{W(k)}$ be the category of complete Noetherian commutative local $W(k)$-algebras $(A, \m_{A})$ with residue field $k$.
\end{defn}

\subsection{Residual Representation}{\label{general residual set up}}
Consider the following residual representation: $\overline{\sigma}: \G \to \GL_{n}(k)$,
\begin{equation}{\label{reducible}}
\overline{\sigma} = 
\begin{pmatrix}
    \rho_1 & a & b \\ 
    0 & \rho_2 & c \\ 
    0 & 0 & \rho_3 \\ 
\end{pmatrix},
\end{equation}
where $\rho_i|_{G_{F_{v}}}: G_{F_{v}} \to \GL_{n_{i}}(k)$ are absolutely irreducible for each $i$ and $\rho_i|_{G_{F_{v}}} \not \cong \rho_j|_{G_{F_{v}}}$ for $i \ne j$ and $v|p$. 

We recall the definition of anti-involution below.
\begin{defn}
An \textbf{anti-involution} $\tau$ on a ring \( R \) is a map \( \tau: R \to R \) such that:
\begin{enumerate}
    \item \(\tau(ab) = \tau(b)\sigma(a)\) for all \( a, b \in R \),
    \item \(\tau(\tau(a)) = a\) for all \( a \in R \).
\end{enumerate}
\end{defn}

\noindent Suppose $\overline{\sigma}$ satisfies:
\begin{itemize}
    \item $\overline{\sigma}$ is Fontaine--Laffaille at $p$;
    \item $\begin{pmatrix}
    \rho_1 & a \\
    0 & \rho_2
    \end{pmatrix}
    $
    and 
    $
    \begin{pmatrix}
        \rho_2 & c \\
        0 & \rho_3
    \end{pmatrix}
    $
    are non-split extensions;
    \item There is some anti-involution $\tau$ of $\R[G_{\Sigma}]$ permuting the Jordan-H\"{o}lder factors of $\overline{\sigma}$, i.e.,
    $$
        \tau(\rho_1) = \rho_3, \tau(\rho_2) = \rho_2, \tau(\rho_3) = \rho_1,
    $$
    for some $\R \in \text{ob}(\hat{\C}_{W(k)})$, and $\overline{\sigma}$ is $\tau$\textbf{-self-dual}, i.e., $\tr \overline{\sigma} = \tr \overline{\sigma} \circ \tau$.
\end{itemize}
These assumptions are satisfied in many cases. In particular, the representations attached to abelian surfaces with $p$-rational isogenies and a polarization degree prime to $p$ or automorphic forms congruent to certain Ikeda lifts are examples of such a setup. See Section \ref{Abelian surface}, \ref{Ikeda lifts} for more details.

\subsection{Main Assumptions}
\begin{assumption}
{\label{main assumptions}}
Throughout this paper, we will assume
\begin{enumerate}
    \item $H^{1}_{\Sigma} (F, \Hom(\rho_3, \rho_1)) = 0$;
    \item $\dim H^{1}_{\Sigma} (F, \Hom(\rho_2, \rho_1)) = \dim H^{1}_{\Sigma} (F, \Hom(\rho_3, \rho_2)) = 1$;
    \item $\dim H^{1}_{\Sigma} (F, \Hom(\rho_1, \rho_2)) = 1$;
    \item if $n_i$ is the dimension of $\rho_i$, and $R_{i,\FL}$ denotes the universal deformation ring for $\rho_i: \G \to \GL_{n_{i}}(k)$ with the condition that the deformation is Fontaine--Laffaille at $p$, then $R_{i,\FL} = \mO$. Equivalently, there is a unique such deformation of $\rho_i$ to $\mO$ up to a strict equivalence class. Denote $\Tilde{\rho}_i$ as the unique deformation of $\rho_i$ to $\GL_{n_{i}}(\mO)$.
\end{enumerate}
\end{assumption}

\begin{rmk}
In application, if $\overline{\sigma}$ is $\tau$-self-dual for some anti-involution $\tau$, one can achieve $H^{1}_{\Sigma} (F, \Hom(\rho_3, \rho_2)) = H^{1}_{\Sigma} (F, \Hom(\rho_1, \rho_2))$ (cf. \cite{bellaiche2009families} Lemma 1.8.5(ii) and our Proposition \ref{Ext with C pseudo}). See Section (\ref{Abelian surface})(\ref{Ikeda lifts}).
\end{rmk}

\begin{rmk}{\label{comments on Selmer group sizes}}
In our experience, we expect the vanishing condition (1) and cyclicity conditions of $H^1_{\Sigma}(\cdot)$ to hold often. For example, suppose one considers an irreducible Galois representation $\overline{\rho}_{p,E}: G_{\Q} \to \GL_2(\F)$ attached to an elliptic curve $E$ that does not have rational $p$-torsion, and suppose $H^{1}_{\Sigma} (\Q, \overline{\rho}_{p,E}) = H^{1}_f (\Q, \overline{\rho}_{p,E})$, then in this case the $p$-Selmer group of $E$ satisfies
$$
\text{Sel}_p(E)[p] = H^{1}_{\Sigma} (\Q, \overline{\rho}_{p,E}) 
$$
and fits into the exact sequence 
\begin{equation}{\label{eq 2.6}}
    0 \to E[\Q]/pE[\Q] \to \text{Sel}_p(E)[p] \to \Sha(E)[p] \to 0.
\end{equation}

If $\Sha(E)$ has trivial $p$-torsions, in particular, if $\Sha(E)$ itself is trivial, then $\text{Sel}_p(E)[p]$ vanishes if rank($E$) = $0$ and $\dim \text{Sel}_p(E)[p] = 1$ if rank($E$) = $1$. Using all curves with conductors $< 500000$ from \cite{lmfdb}, there are around $30.9\%$ curves with $\Sha(E)$ trivial and rank $0$, and $49.4\%$ curves with $\Sha(E)$ trivial and rank $1$. About $92\%$ curves have trivial $\Sha(E)$. Overall, it is proved in \cite{bhargava2013average} that at least 83.75\% curves have rank $0$ or $1$ and it is conjectured, originating in works of \cite{katz2023random}, \cite{katz2023random} and later in \cite{bhargava2013modeling}, that $50\%$ of all elliptic curves over $\Q$ have rank $0$ and $50\%$ have rank $1$. 

In general, the Poonen–Rains heuristics \cite[Conjecture 1.1(b)]{poonen2012random} conjectures that for the average size of the $\text{Sel}_p(E)$ is $p + 1$ for any prime number $p$, which is verified for a few cases, for example, \cite{bhargava2013average}. 
\end{rmk}

\subsection{Uniqueness of the Residual Representation}
The goal of this section is to show that the residual representation is unique up to isomorphism under Assumptions \ref{main assumptions} (1)(2). 

\begin{lem}{\label{Selmer exact sequence}}
There is an exact sequence
\begin{equation}\label{Selmer exact}
    H^{1}_{\Sigma}(F, \Hom(\rho_3, \rho_1)) \to H^{1}_{\Sigma}(F,\Hom(\rho_3,\begin{pmatrix}
    \rho_1 & a \\
    0 & \rho_2
    \end{pmatrix})) \to H^{1}_{\Sigma}(F, \Hom(\rho_3, \rho_2)).
\end{equation}
\end{lem}
\begin{proof}
As $\overline{\sigma}$ is Fontaine--Laffaille at $p$ and the essential image of $\mathbf{G}$ is closed under taking subquotients,  
$\begin{pmatrix}
\rho_1 & a \\
0 & \rho_2
\end{pmatrix}$
is Fontaine--Laffaille at $p$, and $\rho_1, \rho_2, \rho_3$ are all Fontaine--Laffaille at $p$. Consider the exact sequence 
$$
0 \to \rho_1 \to 
\begin{pmatrix}
\rho_1 & a \\
0 & \rho_2
\end{pmatrix} 
\to \rho_2 \to 0.
$$
and the corresponding exact sequence of the pre-images via $\mathbf{G}$:
\begin{equation}{\label{eq 2.6}}
    0 \to M_1 \to M \to M_2 \to 0
\end{equation}

Let $M_3$ be the filtered Dieudonn\'{e} module corresponding to $\rho_3$.
If we apply the homological functor $\Hom_{\mathcal{MF}_{\mO}}(M_3,-)$ to equation \eqref{eq 2.6}, then we get the following long exact sequence in $\mathcal{MF}_{\mO}$ \cite[Section 13.27]{stacks-project}:
\begin{align*}
    0 &\to \Hom_{\mathcal{MF}_{\mO}}(M_3,M_1) \to \Hom_{\mathcal{MF}_{\mO}}(M_3,M) \to \Hom_{\mathcal{MF}_{\mO}}(M_3,M_2) \\
    &\to \Ext^1_{\mathcal{MF}_{\mO}}(M_3,M_1) \to \Ext^1_{\mathcal{MF}_{\mO}}(M_3,M) \to \Ext^1_{\mathcal{MF}_{\mO}}(M_3,M_2) \\
    &\to \Ext^2_{\mathcal{MF}_{\mO}}(M_3,M_1) \to ...
\end{align*}
where $\Ext^i_{\mathcal{MF}_{\mO}}(M_3,-)$ is the set of degree $i$ Yoneda extension of $M_3$ by $(-)$ (Lemma 13.27.4 in \cite[section 13.27]{stackexchange}).
 
As $\rho_3|_{G_{F_{v}}} \not \cong \rho_2|_{G_{F_{v}}}$ for all $v \mid p$, one has $\Hom_{k[G_{F_{v}}]}(\rho_3,\rho_2) = 0$. It follows from full faithfulness of $\mathbf{G}$ that $\Hom_{\mathcal{MF}_{\mO}}(M_3,M_2) = 0$. Lemma 4.4 in \cite{bloch2007functions} gives $\Ext^2_{\mathcal{MF}_{\mO}}(M_3,M_1) = 0$. Thus we have a short exact sequence 
\begin{equation}{\label{ext exact}}
    0 \to \Ext^1_{\mathcal{MF}_{\mO}}(M_3,M_1) \to \Ext^1_{\mathcal{MF}_{\mO}}(M_3,M) \to \Ext^1_{\mathcal{MF}_{\mO}}(M_3,M_2) \to 0.
\end{equation}

Following the argument in (\cite[p.~712]{diamond2004tamagawa}), for $i = 3$, $j \in \{1,2\}$, let $\Ext^1_{k[G_{F_{v}}],\FL}(\rho_i,\rho_j)$ be the image of $\Ext^1_{\mathcal{MF}_{\mO}}(M_i,M_j)$ in $\Ext^1_{k[G_{F_{v}}]}(\rho_i,\rho_j) $. Then we have the following chain of isomorphisms:
\begin{align*}
    \Ext^1_{\mathcal{MF}_{\mO}}(M_i,M_j) &\xrightarrow[\cong]{\phi} \Ext^1_{k[G_{F_{v}}],\FL}(\rho_i,\rho_j)  \cong \Ext^1_{k[G_{F_{v}}],\FL}(k,\Hom_k(\rho_i,\rho_j)) \\
    &\cong H^1_f(F_{v},\Hom_k(\rho_i,\rho_j)).
\end{align*}
The map $\phi$ is an isomorphism because $\mathbf{G}$ is fully faithful and exact. Furthermore the last isomorphism follows from the definition of $H^1_f(F_v,\cdot)$ for a $p$-adic $G_{F_{v}}$-module. The same argument also gives 
$$\Ext^1_{\mathcal{MF}_{\mO}}(M_3,M) \cong H^1_f(F_v, \Hom(\rho_3,\begin{pmatrix}
\rho_1 & a \\
0 & \rho_2
\end{pmatrix})).
$$

Thus (\ref{ext exact}) becomes
\begin{equation}{\label{induced}}
   0 \to H^1_f(F_v,\Hom_k(\rho_3,\rho_1)) \to H^1_f(F_v, \Hom(\rho_3,\begin{pmatrix}
\rho_1 & a \\
0 & \rho_2
\end{pmatrix})) \to H^1_f(F_v,\Hom_k(\rho_3,\rho_2)) \to 0. 
\end{equation}

By a similar argument as the proof of Lemma 5.3 in \cite{berger2013deformation}, we get that (\ref{induced}) is induced by the short exact sequence
\begin{equation*}\label{short exact}
    0 \xrightarrow[]{} \Hom(\rho_3, \rho_1) \xrightarrow[]{i} \Hom(\rho_3,\begin{pmatrix}
    \rho_1 & a \\
    0 & \rho_2
    \end{pmatrix}) \xrightarrow[]{j} \Hom(\rho_3, \rho_2) \to 0.
\end{equation*}
i.e.
$
H^1_f(F_v,\Hom_k(\rho_3,\rho_1)) = i_{*}^{-1} H^1_f(F_v, \Hom(\rho_3,\begin{pmatrix}
\rho_1 & a \\
0 & \rho_2
\end{pmatrix})), H^1_f(F_v,\Hom_k(\rho_3,\rho_2)) = j_{*} H^1_f(F_v, \Hom(\rho_3,\begin{pmatrix}
\rho_1 & a \\
0 & \rho_2
\end{pmatrix}))
$. Then by Lemma 2.3 in \cite{berger2020deformations} we obtain \eqref{Selmer exact}. 
\end{proof}

\begin{prop}{\label{Uniqueness of residual}}
If  $\overline{\sigma}_1 =  
\begin{pmatrix}
\rho_1 & a & b \\ 
0 & \rho_2 & c \\ 
0 & 0 & \rho_3
\end{pmatrix}$, and 
$\overline{\sigma}_2 =  
\begin{pmatrix}
\rho_1 & a' & b' \\ 
0 & \rho_2 & c' \\ 
0 & 0 & \rho_3
\end{pmatrix}$ are both Fontaine--Laffaille at $p$ and 
$
\begin{pmatrix}
    \rho_1 & * \\
    0 & \rho_2
\end{pmatrix}
$
and 
$
\begin{pmatrix}
    \rho_2 & * \\
    0 & \rho_3
\end{pmatrix}
$
are all non-split, then Assumptions \ref{main assumptions} (1)(2) imply
$\overline{\sigma}_1 \cong \overline{\sigma}_2$.
\end{prop}
\begin{proof}
Both 
$
\begin{pmatrix}
    \rho_1 & a \\
    0 & \rho_2
\end{pmatrix}
$
and 
$
\begin{pmatrix}
    \rho_1 & a' \\
    0 & \rho_2
\end{pmatrix}
$
give rise to non-trivial elements in $H^{1}_{\Sigma}(F, \Hom(\rho_2,\rho_1))$. By Assumption \ref{main assumptions} (2) that $\dim_{k} H^{1}_{\Sigma}(F, \Hom(\rho_2,\rho_1)) = 1$, we can assume that $a = \alpha a'$, where $\alpha \in k^{\times}$. Similarly, we can assume that $c = \beta c'$ by $\dim H^{1}_{\Sigma}(F, \Hom(\rho_3,\rho_2)) = 1$, for some $\beta \in k^{\times}$. 

As
$$
\begin{pmatrix}
    I_1 & 0 & 0\\
    0 & \alpha I_2 & 0 \\
    0 & 0 & \alpha \beta I_3 \\
\end{pmatrix}
\overline{\sigma}_2
\begin{pmatrix}
    I_1 & 0 & 0\\
    0 & \alpha I_2 & 0 \\
    0 & 0 & \alpha \beta I_3 \\
\end{pmatrix}^{-1}
=
\begin{pmatrix}
    \rho_1 & a & \alpha \beta^{-1} b\\
    0 & \rho_2 & c \\
    0 & 0 & \rho_3 \\
\end{pmatrix},
$$
we can in fact assume that $a' = a$ and $c' = c$.

Under Assumption \ref{main assumptions} (1) that $H^{1}_{\Sigma}(F, \Hom(\rho_3, \rho_1)) = 0$, by Lemma \ref{Selmer exact sequence}, we get
$$
H^{1}_{\Sigma}(F,\Hom(\rho_3,\begin{pmatrix}
\rho_1 & a \\
0 & \rho_2
\end{pmatrix})) \hookrightarrow H^{1}_{\Sigma}(F, \Hom(\rho_3, \rho_2))
$$
is an injection, where the latter Selmer group is assumed to be 1-dimensional. Then since both $\overline{\sigma}_1, \overline{\sigma}_2$ are non-zero elements in $H^{1}_{\Sigma}(F,\Hom(\rho_3,\begin{pmatrix}
\rho_1 & a \\
0 & \rho_2
\end{pmatrix}))$, we have $\overline{\sigma}_1 \cong \overline{\sigma}_2$.
\end{proof}

\subsection{Deformation Problem}{\label{deformation problem}}


\begin{lem}\label{scalar centralizer}
$\overline{\sigma}$ has a scalar centralizer.
\end{lem}
\begin{proof}
Suppose
$\begin{pmatrix}
    A & B & C\\
    D & E & F\\
    G & H & I
\end{pmatrix}
\in \GL_n(k)
$
is an element of the centralizer of $\overline{\sigma}$, where $A, ..., I$ are matrices of appropriate sizes, then 
$$
\begin{pmatrix}
    A & B & C\\
    D & E & F\\
    G & H & I
\end{pmatrix}
\begin{pmatrix}
        \rho_1 & a & c \\
        0 & \rho_2 & b \\
        0 & 0 & \rho_3
\end{pmatrix}
= 
\begin{pmatrix}
        \rho_1 & a & c \\
        0 & \rho_2 & b \\
        0 & 0 & \rho_3
\end{pmatrix}
\begin{pmatrix}
    A & B & C\\
    D & E & F\\
    G & H & I
\end{pmatrix},
$$
gives 
\begin{equation}{\label{eq1}}
\begin{pmatrix}
    A \rho_1 & A a + B \rho_2 & *\\
    D \rho_1 & D a + E \rho_2 & *\\
    G \rho_1 & G a + H \rho_2 & *
\end{pmatrix}
= 
\begin{pmatrix}
    \rho_1 A + aD + cG & \rho_1 B + a E + c H & *\\
    \rho_2 D + b G & \rho_2 E + b H & * \\
    \rho_3 G & \rho_3 H & *
\end{pmatrix}.
\end{equation}
Position (3,1) gives $G \rho_1 = \rho_3 G$. We must have $G = 0$ since $\rho_1 \not \cong \rho_3$. 

For similar reasons, we get $D = 0$ and $H = 0$. Now \eqref{eq1} becomes 
\begin{equation}\label{eq 1}
\begin{pmatrix}
    A \rho_1 & A a + B \rho_2 & A c + B b + C \rho_3 \\
    0 & E \rho_2 & E b + F \rho_3\\
    0 & 0 & I \rho_3
\end{pmatrix}
= 
\begin{pmatrix}
    \rho_1 A & \rho_1 B + a E & \rho_1 C + a F + c I\\
    0 & \rho_2 E & \rho_2 F + b I \\
    0 & 0 & \rho_3 I
\end{pmatrix} 
\end{equation}

Note that because $\rho_i$'s are absolutely irreducible for all $i$, their centralizers are scalar matrices by Schur's lemma. Thus, the diagonal entries show that $A,E,I$ are scalar matrices.

Since $A, E$ are scalar matrices, $A = \alpha_{A} I_{n_{1}}, E = \alpha_{E} I_{n_{2}}$ for some $\alpha_{A}, \alpha_{E} \in k^{\times}$ and correct size of identity matrices $I_{n_{1}}, I_{n_{2}}$. 

We claim $\alpha_{A} = \alpha_{E}$. Suppose $\alpha_{A} \ne \alpha_{E}$, then
$$
a = (A-E)^{-1} (\rho_1 B - B \rho_2).
$$
As $\begin{pmatrix}
    \rho_1 & a \\
    0 & \rho_2 \\
\end{pmatrix}$ is a non-split extension, $a \rho_2^{-1} \in H^1(\G, \Hom(\rho_2, \rho_1))$ should not a coboundary. However, if $g \in \G$, the action of $g$ on $a \in \Hom(\rho_2, \rho_1)$, $g \dot a (x) = g \cdot a(g^{-1} x) = \rho_1(g) a(x) \rho_2^{-1}(g)$ and 
$$
a \rho_2^{-1} (g) = (A-E)^{-1} (\rho_1 (g) B \rho_2^{-1}(g) - B) = g \cdot \varphi - \varphi,
$$
would a coboundary, for $\varphi = (A-E)^{-1} B \in \Hom(\rho_2, \rho_1)$. Thus $\alpha_{A} = \alpha_{E} = \alpha \in k^{\times}$.

Then the entries corresponding to the position (1,2) reduce to $B \rho_2 = \rho_1 B$. As $\rho_2 \not \cong \rho_3$, we get $B = 0$. Similarly, $I = \alpha I_{n_{3}}$ and $F = 0$. Then position (1,3) now gives $C \rho_3 = \rho_1 C$ and $C = 0$. Hence $\overline{\sigma}$ has a scalar centralizer. 
\end{proof}

Let $\mathcal{R}$ be an object in $\hat{\mathcal{C}}_{W(k)}$. Let $M$ be a rank $n$ free $\R$-module, $\mathcal{D}(\R)$ denote the set of deformations $\sigma: G_{\Sigma} \to \GL_{\R}(M)$ of $\overline{\sigma}$ with the following deformation condition:
\begin{enumerate}
    \item $\sigma$ is Fontaine-Laffaile at $p$.
\end{enumerate}

\begin{prop}
The functor $\mathcal{D}$ is representable by a universal deformation ring $R_{\FL}$ with corresponding universal deformation $\sigma_{\FL}$.
\end{prop}
\begin{proof}
Being Fontaine--Laffaille at $p$ is a deformation condition and the endomorphism ring of the Galois module associated to $\overline{\sigma}$ is $k$ by Lemma \ref{scalar centralizer}. It follows from \cite{ramakrishna1993variation} that the problem is representable.
\end{proof}

We proceed to consider an extra condition. Let $\mathcal{D}'(\R)$ denote the set of deformations $\sigma: G_{\Sigma} \to \GL_{n}(\mathcal{R})$ of $\overline{\sigma}$ with conditions:
\begin{enumerate}
    \item $\sigma$ is Fontaine-Laffaile at $p$.
    \item $\sigma$ is $\tau$-self-dual for some anti-involution $\tau$.
\end{enumerate}

\begin{prop}
The functor $\mathcal{D}'$ is representable by a universal deformation ring $R'$.
\end{prop}
\begin{proof}
$R'$ is the quotient of $R_{\FL}$ by the ideal generated by $\{\tr(\sigma_{\FL}(g)) - \tr(\sigma_{\FL} \circ \tau)(g)) | g \in G_{\Sigma}\}$.
\end{proof}

Furthermore, we are interested in the reduced universal deformation ring $R$. 
\begin{defn}
    $R$ is the quotient of $R'$ by its nilradical and $\sigma$ is its corresponding universal deformation.
\end{defn}


\section{Principality of The Reducibility Ideal}
The goal of this section is to show that, under Assumptions (\ref{main assumptions})(1)(2)(3), the reducibility ideal of \( R \) is principal by applying methods from the theory of pseudorepresentations. 

The section begins by discussing previous results in deformation theory and introducing pseudorepresentations, the reducibility ideal, pseudodeformations, and GMAs. Since we use the theory of pseudodeformations rather than directly studying them, we then identify the GMA structure of \( R[\G]/\ker \sigma \) in our setting, and use it to determine the structure of the reducibility ideal of \( R \). Next, we provide a new result that relates the GMA structure of a universal pseudodeformation to Selmer groups, along with a parallel result for \( \sigma \). Then, we use bounds on the Selmer groups and the anti-involution \( \tau \) to establish that the reducibility ideal is principal. Finally, we state sufficient conditions such that $\ker \sigma = \ker (\det \circ \sigma)$.

If the residual representation in interest has two Jordan-H\"{o}lder factors $\rho_1, \rho_2$, one can show that the reducibility ideal is principal if 
\begin{equation*}\tag{$\star$}
    \dim_k H^1_{\Sigma}(F,\Hom(\rho_j, \rho_i)) = 1, 
\end{equation*}
for $i,j \in \{1,2\}, i \ne j$. However, in the context of three Jordan-H\"{o}lder factors, the structure of the total reducibility ideal  $I^{tot}$ is more complicated. These assumptions in Assumptions \ref{main assumptions} (1)-(3) and self-duality of $\sigma$ are key to simplifying the expression of $I^{\tot}$ and achieving principality. They can be thought of as analogous to $(\star)$ in the case of three Jordan-H\"{o}lder factors.

Bella\"{i}che and Chenevier \cite{bellaiche2009families} first studied the theory of pseudocharacters and introduced generalized matrix algebras (GMAs). Later, Chenevier \cite{chenevier2014p} generalized pseudocharacters to pseudorepresentations (\textit{determinants} in his terminology), extending the base ring to an arbitrary ring. Furthermore, Wake and Wang-Erickson \cite{wake2019deformation} studied pseudodeformation conditions for pseudorepresentations. Their work regarding representability of the pseudodeformation functor with conditions extends the result of \cite{ramakrishna1993variation} from representations to pseudorepresentations. 
We will largely follow notation and definitions in \cite{wake2019deformation} and \cite{chenevier2014p}.

\subsection{Previous results in deformation theory}
In this section, we review results from \cite{chenevier2014p} \cite{wake2019deformation} regarding pseudorepresentations, Cayley-Hamilton, reducibility ideal, deformation conditions for pseudorepresentations, and GMAs.

\subsubsection{\textsf{Pseudorepresentations}}
Below we review the definition of pseudorepresentations and compare it to pseudocharacters. Let $\R \in \text{ob}(\hat{\mathcal{C}}_{W(k)})$.

\begin{defn}
Let $E$ be an $\R$-algebra.
\begin{enumerate}
    \item There is a functor $\underline{E}: A \to  E \otimes_{\R} A$ from the category of $\R$-algebras to sets for each commutative $\R$-algebra $A$. A \textbf{pseudorepresentation} $D: E \to \R$, or $(E,D)$ is a natural transformation from $\underline{E}$ to $\underline{\R}$ whose data consists of a function $D_{A}: E \otimes_{\R} A \to A$ for each $A$ that satisfies the followings:
    \begin{enumerate}
        \item $D$ is multiplicative, $D_{A}(1) = 1, D_{A}(xy) = D_{A}(x) D_{A}(y)$ for all $x,y \in E \otimes_{\R} A$;
        \item $D$ has degree $d \ge 1$, i.e., $D_{A}(bx) = b^{d} D_{A}(x)$ for all $x \in E \otimes_{\R} A$ for all $b \in A.$
    \end{enumerate}
    \item When $E = \R[G_{\Sigma}]$, we say that $D: \R[G_{\Sigma}] \to \R$, is a \textbf{pseudorepresentation of} $G_{\Sigma}$. 
    $d$ is the \textbf{dimension} of $D$, and $\R$ the \textbf{scalar ring} of $(E,D)$. 
    
    \item If $D: E \to \R$ is a pseudorepresentation of degree $d$, and $x \in E$, the \textbf{characteristic polynomial} $\chi_{D}(x,t) \in \R[t]$ is defined as $\chi_{D}(x,t) = D_{\R[t]}(t-x)$. It is monic and its degree is $d$. \textbf{Trace} $T_{D}(x)$ is the additive inverse of the coefficient of $t^{d-1}$ in $\chi_{D}(x,t)$. 

    \item Let $D \otimes_{\R} k: E \otimes_{\R}  k \to k$ be the pseudopresentation of $E \otimes_{\R} k$ defined by $(D \otimes_{\R} k)_A(x) = D_{A}(x)$ for each $k$-algebra $A$. 

\end{enumerate}
\end{defn}

\begin{rmk}
   The map $T_{D}: E \to \R$ is a pseudocharacter in the sense of \cite{bellaiche2009families} (cf. \cite[Lemma. 1.12]{chenevier2014p}) and the definition of a pseudorepresentation above is the same as the definition of a \textit{determinant} in \cite{chenevier2014p}. Pseudorepresentation theory extends the theory of pseudocharacters as it works for an arbitrary ring, i.e., contrary to the case of pseudocharacters, there is no need for the assumption that $d! \in \R^{\times}$. There is a bijection between the sets of pseudocharacters and the set of pseudorepresentations if $(2d)! \in \R^{\times}$ (cf. \cite[Prop. 1.29]{chenevier2014p}). \cite{ophir2024relation} recently has improved the result to require only $d! \in \R^{\times}$.
\end{rmk}



\begin{defn}
A pseudorepresentation $D: E \to \R$ is \textbf{Cayley-Hamilton} if $E$ is finitely generated as an $\R$-algebra and for every commutative $\R$-algebra $A$ and $x \in E \otimes_{\R} A$, one has $\chi_{D}(x,x) = 0.$ When $D: E \to \R$ is Cayley-Hamilton, we call $(E,D)$ a \textbf{Cayley-Hamilton $\R$-algebra}. Define a \textbf{Cayley-Hamilton representation} of $G_{\Sigma}$ over $\R$ as a triple $(E,\rho,D)$, where $(E,D)$ is a Cayley-Hamilton $\R$-algebra and $\rho: G_{\Sigma} \to E^{\times}$ is a group homomorphism. A \textbf{morphism of Cayley-Hamilton representations} of $\G$, $(E, \rho, D) \to
(E', \rho',D')$, is a map of Cayley-Hamilton algebras $(E,D) \to (E',D')$ such that $\rho' = (E \to E') \circ \rho$.
\end{defn}

\subsubsection{\textsf{Reducibility ideal}}
For simplicity, from now, we demonstrate the results by restricting to have $3$ Jordan-H\"older factors as in our setting, i.e., we consider a pseudorepresentation $\D = \det \circ \overline{\sigma} = \det (\rho_1 \oplus\rho_2 \oplus \rho_3): k[G_{\Sigma}] \to k$. We note that the results below apply similarly if $\overline{\sigma}$ has $n$ Jordan-H\"older factors.

Following \cite[Section 1.5]{bellaiche2009families} and \cite{chenevier2014p}, we have the following definition of reducibility ideal and reducible pseudorepresentation.

\begin{defn}{\label{red def}} 
Recall that $\R$ is an object in $\hat{\mathcal{C}}_{W(k)}$. Let $\cP = (\cP_1, ..., \cP_s)$ be a partition of $\{1,2,3\}$ and $D': \R[G_{\Sigma}] \to \R$ such that $D' \otimes_{\R} k \cong \D$. The \textbf{reducibility ideal} $I_{\cP,\R}$ of $\R$ is the ideal of $\R$ such that for every ideal $J$ of $\R$, $I_{\cP,\R} \subseteq J$ if and only if $D'$ is $\cP$-\textbf{reducible}, i.e., there exists pseudorepresentations $D^{'}_1, ..., D^{'}_s: \R[G_{\Sigma}] \otimes_{\R} \R/J \to \R/J$ such that for each $\R/J$-algebra $A$ and $x \in E \otimes_{\R/J} A$, 
\begin{enumerate}
    \item $T_{D'} \otimes \R/J = \sum_{i=1}^{s} T_{D'_{i}}$,
    \item $T_{D^{'}_{i}} \otimes k =\tr (\bigoplus_{j \in \cP_{i}} \rho_j)$ for $i \in \{1,...,s\}$.
\end{enumerate}
\end{defn}

\subsubsection{\textsf{Deformation conditions for pseudorepresentations}}
\vspace*{0.1cm}

In \cite{wake2019deformation}, the authors generalized Ramakrishna's result \cite{ramakrishna1993variation} and showed the representability of the pseudodeformation functor with pseudodeformation conditions. Particularly, in this paper, we are interested in the pseudodeformation condition being Fontaine--Laffaille at $p$ and demonstrate their result in this case.

First, we give the definition of a pseudorepsentation that is Fontaine--Laffaille at $p$.

\begin{defn}
Let $(E,\rho,D': \R[G_{\Sigma}] \to \R)$ over $\R$ be Cayley-Hamilton representation where $E$ is a finitely generated $\R$-module. We say that $E$ is\textbf{ Fontaine--Laffaille at }$p$ as a $\R[G_{\Sigma}]$-module if $E|_{\R[G_{F_{v}}]}$ is in the essential image of $\mathbf{G}$. Furthermore, one says $D'$ is \textbf{Fontaine--Laffaille at} $p$ if there is a Cayley-Hamilton representation $(E,\rho,D'': \R[G_{\Sigma}] \to \R)$ such that $E$ is Fontaine--Laffaille at $p$ as a $\R[G_{\Sigma}]$-module and $D' = D'' \circ \rho$.
\end{defn}

Let PsDef$^{\FL}_{\Bar{D}}: \hat{\mathcal{C}}_{W(k)} \to \text{Sets}$ be the Fontaine--Laffaille at $p$ pseudodeformation functor of $\D$, i.e.,
\begin{align*}
\text{PsDef}_{\Bar{D}}^{\FL}: \Big\{ A \in \hat{\mathcal{C}}_{W(k)} \Big\}
\to 
\Big\{&D': A[G_{\Sigma}] \to A \text{ s.t. } D' \otimes_{A} k \cong \D
\\
&\text{such that $D'$ is Fontaine--Laffaille at $p$} \Big\} 
\end{align*}
Since the Fontaine--Laffaille at $p$ condition is a stable condition in the sense of Definition 2.3.1 in \cite{wake2019deformation}, Theorem 2.5.5 in there shows that PsDef$^{\FL}_{\Bar{D}}$ is a representable functor. 

Let $(E_{\D}^{\FL},\rho_{\D}^{\FL},D_{\D}^{\FL}: R_{\D}^{\FL}[G_{\Sigma}] \to R_{\D}^{\FL})$ be the universal pseudodeformation. Then $E_{\D}^{\FL}$ has data $\E^{\FL}_{\D} = \{e^{\FL}_{\D,i}, \psi^{\FL}_{\D,i}\}$ GMA structure
$$
E_{\D}^{\FL} \cong 
\begin{pmatrix}
    M_{1,1}(R_{\D}^{\FL}) & M_{1,2}(\A_{1,2}^{ \FL}) & M_{1,3}(\A_{1,3}^{ \FL})\\
    M_{2,1}(\A_{2,1}^{ \FL}) & M_{2,2}(R_{\D}^{\FL}) & M_{2,3}(\A_{2,3}^{ \FL})\\
    M_{3,1}(\A_{3,1}^{ \FL}) & M_{3,2}(\A_{3,2}^{ \FL}) & M_{3,3}(R_{\D}^{\FL})
\end{pmatrix}.
$$
Many details concerning GMAs can be found in \cite[~Section 1.3]{bellaiche2009families}.
\vspace*{0.1cm}

Lastly, we introduce the definition and properties of Cayley-Hamilton $G_{\Sigma}$-modules. A \textbf{Cayley-Hamilton $G_{\Sigma}$-module} is an $E$-module $N$ for some Cayley-Hamilton representation $(E,\rho,D')$ of $G_{\Sigma}$. $(E,D')$ is the Cayley-Hamilton algebra of $N$. We say $N$ is faithful if it is faithful as a $E$-module. The next Theorem shows that a faithful $G_{\Sigma}$-module is Fontaine--Laffaille at $p$ if its Cayley-Hamilton algebra is Fontaine--Laffaille at $p$.

\begin{thm}{\label{N has C}}
Let $N$ be a faithful Cayley-Hamilton $G_{\Sigma}$-module with Cayley-Hamilton $\R$-algebra $(E,D')$. Then the following are equivalent:
\begin{enumerate}
    \item $N$ is Fontaine--Laffaille at $p$ as an $\R[G_{\Sigma}]$-module,
    \item $(E,\rho,D')$ is Fontaine--Laffaille at $p$ as a Cayley-Hamilton representation.
\end{enumerate}
\end{thm}
\begin{proof}
This is Theorem 2.6.4 in \cite{wake2019deformation}.
\end{proof}

\subsection{New results in deformation theory}
In this subsection, we return to our setting and study the GMA structures of \( R[G_{\Sigma}]/\ker \sigma \) and \( R[G_{\Sigma}]/\ker D \), where $D = \det \circ \sigma$, and use it to identify the structure of the total reducibility ideal of \( R \). We note that the results in \cite{wake2019deformation} cannot be directly applied, as even though we rely on the theory of pseudodeformations, \( \sigma \) is not a pseudodeformation, and new results must be provided. We then prove Proposition \ref{Ext with C pseudo} that links \( \A_{i,j}^{\FL} \) in \( E_{\D}^{\FL} \) to Selmer groups, along with a parallel result Proposition \ref{Ext with C} for \( R[G_{\Sigma}]/\ker \sigma \). Then, we show that the principality of the Selmer groups and the anti-involution \( \tau \) guarantees that the reducibility ideal is principal. Finally, we identify sufficient conditions when $\ker \sigma = \ker D$.

\subsubsection{\textsf{Generalized Matrix Algebras (GMAs)}}
As before, let $R$ be the universal deformation ring in our setting and $\sigma$ its corresponding universal deformation. Let $D = \det \circ \sigma: R[G_{\Sigma}] \to R$ be the associated universal pseudorepresentation.

We define $\ker D$ following \cite[\S 1.17]{chenevier2014p}:
$$
 \ker(D) = \{x \in E |\text{for all $R$-algebras A, $\forall a \in A$, $\forall m \in E \otimes A$, } D(x \otimes a + m) = D(m) \}.
$$
Following \cite[\S 1.17]{chenevier2014p}, let CH($D$) $\subset R[G_{\Sigma}]$ be the ideal of $R[G_{\Sigma}]$ generated by the coefficients of $\chi_D(x,t) \in R[t]$ with $x \in R[G_{\Sigma}]$. In particular, if CH($D) = 0$, then $D$ is Cayley-Hamilton. 

Let $\star \in \{\sigma, D\}$ and $E^{\star} = R[G_{\Sigma}]/\ker \star$. As $\ker \sigma \subset \ker D$, we have a canonical $R$-algebra surjection $\phi: E^{\sigma} \twoheadrightarrow E^{D}$. The pair $(E^{D},D)$ is a Cayley-Hamilton $R$-algebra as $ \text{CH}(D) \subset \ker D$ (cf. \cite[Lemma 1.21]{chenevier2014p}). By the Cayley-Hamilton theorem, every square matrix over a commutative ring satisfies its own characteristic equation, and $(E^{\sigma},D)$ is also a Cayley-Hamilton $R$-algebra.

\begin{lem}{\label{GMA}}
\begin{enumerate}
    \item For $\star \in \{\sigma, D\}$, there are data $\E^{\star} = \{e_i^{\star}, \psi_i^{\star}\}$ such that $E^{\star}$ is GMA in the sense of \cite[\S 1.3]{bellaiche2009families}. More precisely, for a fixed $\star$, $\{e_i^{\star}\}_{i=1}^{3}$ is a set of orthogonal idempotents and $\psi_i^{\star}: e_i^{\star} E^{\star} e_i^{\star} \to M_{i,i}(R)$ is an $R$-algebra isomorphism with $\psi_i^{\star} \otimes k = \rho_i$. $e_i^{\star} E^{\star} e_j^{\star} = M_{i,j}(A_{i,j}^{\star})$. $A_{i,j}^{\star}$ satisfy that 
$$
A_{i,j}^{\star} A_{j,k}^{\star} \subset A_{i,k}^{\star}, \quad A_{i,i}^{\star} \cong R, \quad  A_{i,j}^{\star} A_{j,i}^{\star} \subset \mathfrak{m},
$$
where $\mathfrak{m}$ is the maximal ideal of $R$. Furthermore, 
\begin{equation}{\label{S^T matrix}}
E^{\star} \cong
\begin{pmatrix}
    M_{1,1}(R) & M_{1,2}(A_{1,2}^{\star}) & M_{1,3}(A_{1,3}^{\star})\\
    M_{2,1}(A_{2,1}^{\star}) & M_{2,2}(R) & M_{2,3}(A_{2,3}^{\star})\\
    M_{3,1}(A_{3,1}^{\star}) & M_{3,2}(A_{3,2}^{\star}) & M_{3,3}(R)
\end{pmatrix}.
\end{equation}
    \item $\phi(e_i^{\sigma}) = e_i^{D}$ and $\phi(A_{i,j}^{\sigma}) = A_{i,j}^{D}$.
    \item  $A_{i,j}^{D}$ are isomorphic to fractional ideals of $R$. $A_{i,j}^{\sigma}$ are submodules of $R$.
\end{enumerate}
\end{lem}

\begin{proof}
It follows from Theorem 2.22 in \cite{chenevier2014p} that 
$(E^{\star}, T_D)$ are GMAs in the sense of \cite[\S 1.3]{chenevier2014p}. $(E^{\star}, T_D)$ is compatible with $D$ by Proposition 2.23 in \cite{wang2018algebraic} and (1) holds. Note that the proof does not require $d! \in R^{\times}$ as in the proof of Theorem 1.4.4 in \cite{bellaiche2009families}. (2) follows the same argument as in Lemma 2.5 in \cite{berger2013deformation}. For (3), because $R$ is Noetherian and reduced, its total fraction ring is a finite product of fields, and it follows from a similar argument as in Theorem 1.4.4 (ii) in \cite{bellaiche2009families} that $A_{i,j}^{D}$ are fractional ideals. By Proposition 1.3.8 in \cite{bellaiche2009families}, $A_{i,j}^{\sigma}$ are $R$-submodules of $R$.


\end{proof}


\subsubsection{\textsf{Reducibility ideal in $R$}}

The reducibility ideal can be explicitly written. If $\cP = \{1\} \cup \{2\} \cup \{3\}$, we call $I_{\cP,\R}$ the \textbf{total reducibility ideal} of $\R$, and denote it by $I^{\tot}_{\R}$. The following proposition is on $I_{\cP,R}$, but a similar argument could be applied to any $I_{\cP,\R}$. To shorten notation, let $I^{\tot} = I^{\tot}_{R}$.

\begin{prop}{\label{red ideal structure}}
For $\star \in \{\sigma,D\}$, we have
$$I_{\cP,R} = \sum_{\substack{(i,j)\\
                  i,j \text{ not in the same }\cP_l}}
                 A_{i,j}^{\star} A_{j,i}^{\star}.
$$
And $I^{\tot} = A_{1,2}^{\star} A_{2,1}^{\star} + A_{1,3}^{\star} A_{3,1}^{\star} + A_{2,3}^{\star} A_{3,2}^{\star}$.
\end{prop}
\begin{proof}
This follows from a similar argument as in Corollary 6.5 in \cite{berger2020deformations} by considering the Cayley-Hamilton quotients $(E^D, T_D)$ and $(E^{\sigma}, T_D)$  (cf. \cite[Definition 1.2.3]{bellaiche2009families}). Alternatively, the details can also be found in Proposition 1.5.1 in \cite{bellaiche2009families}.
\end{proof}
\begin{rmk}
The Proposition above states the identification of the reducibility ideals in $R$. A similar identification can be realized for any ring $\R$ following a similar proof.
\end{rmk}

\begin{rmk}
While Proposition 1.5.1 in \cite{bellaiche2009families} requires that $d! \in R$, an alternative proof has been provided in Proposition 2.5 in \cite{allen2020automorphy} to remove the $d! \in R$. We would like to thank Tobias Berger for the helpful discussion.
\end{rmk}

\subsubsection{\textsf{$A_{i,j}$ and Selmer groups}}
Let $J \subset \R$ be an ideal such that $I^{\tot}_{\R} \subset J$. Let $\rho_{i,\FL}: G_{\Sigma} \to \GL_{d_{i}}(\R/J)$ be the quotient of $E^{\FL}_{\D} \otimes_{\Rpseudo} \Rpseudo/J$ corresponding to $\rho_i$ and is Fontaine--Laffaille at $p$. For each $v \mid p$, let $G_{F_{v}}$ be the absolute Galois group of $F_{v}$. By an abuse of notation, we denote $\rho_{i,\FL}|_{G_{F_{v}}}$ as $\rho_{i,\FL}$ as well, and let $D_i$ be the filtered Dieudonn\'{e} module such that $\rho_{i,\FL} = \mathbf{G}(D_i)$. Again, let $\Ext^{1}_{\R/J[G_{F_{v}}],\FL} (\rho_{i,\FL}, \rho_{j,\FL})$ be the image of $\Ext^1_{\mathcal{MF_{O}}}(D_i,D_j)$ in $\Ext^{1}_{\R/J[G_{F_{v}}]} (\rho_{i,\FL}, \rho_{j,\FL})$, $i \ne j$. Note that by definition we have
$$
\Ext^{1}_{\R/J[G_{F_{v}}],\FL} (\rho_{j,\FL}, \rho_{i,\FL}) = H^1_f(F_{v},\Hom(\rho_{j,\FL}, \rho_{i,\FL})).
$$

Then, we define the Fontaine--Laffaille at $p$ Ext group $\Ext^{1}_{\R/J[G_{\Sigma}], \FL} (\rho_{j,\FL}, \rho_{i,\FL})$ as
$$
\Ext^{1}_{\R/J[G_{\Sigma}], \FL} (\rho_{j,\FL}, \rho_{i,\FL}) := \ker \bigg(
\Ext^{1}_{\R/J[G_{\Sigma}]} (\rho_{j,\FL}, \rho_{i,\FL})
\to 
\prod_{v \in \Sigma_p}
\frac{\Ext^{1}_{\R/J[G_{F_{v}}]} (\rho_{j,\FL}, \rho_{i,\FL})}{\Ext^{1}_{\R/J[G_{F_{v}}],\FL} (\rho_{j,\FL}, \rho_{i,\FL})}
\bigg),
$$
And by definition,
$$
\Ext^{1}_{\R/J[G_{\Sigma}], \FL} (\rho_{j,\FL}, \rho_{i,\FL}) \cong H^1_{\Sigma}(F,\Hom(\rho_{j,\FL}, \rho_{i,\FL})).
 $$


The next two Propositions identify $\R$-modules in the GMA structure of $E$ with $\R[\G]$-module extensions. They are essential for us to identify the structure of $I^{\tot}$.

\begin{prop}\label{Ext with C pseudo}
(Hom-Ext isomorphism)
Let $J \subset \Rpseudo$ be an ideal of $\Rpseudo$ such that $I^{\tot}_{\Rpseudo} \subset J$. Let $\rho_{i,\FL}: G_{\Sigma} \to \GL_{d_{i}}(\Rpseudo/J)$ be a deformation of $\rho_i$ that is Fontaine--Laffaille at $p$.
Let $\A_{i,j}' = \sum_{k \ne i,j} \A_{i,k}^{\FL} \A_{k,j}^{\FL}$, then there is an isomorphism
\begin{equation}{\label{hom-ext iso}}
    \Hom_{\Rpseudo}(\A_{i,j}^{\FL}/\A_{i,j}',\Rpseudo/J) \cong \Ext^{1}_{\Rpseudo/J[G_{\Sigma}], \FL} (\rho_{j,\FL}, \rho_{i,\FL})
\end{equation} 
\end{prop}
\begin{proof}
First, {as $(\A_{i,j}^{\FL}/\A_{i,j}') \otimes_{\Rpseudo} \Rpseudo/J \cong (\A_{i,j}^{\FL}/\A_{i,j}')/J(\A_{i,j}^{\FL}/\A_{i,j}')$}, we have the following exact sequence 
$$
0 \to J(\A_{i,j}^{\FL}/\A_{i,j}') \to \A_{i,j}^{\FL}/\A_{i,j}' \to (\A_{i,j}^{\FL}/\A_{i,j}') \otimes_{\Rpseudo} \Rpseudo/J  \to 0.
$$
By left exactness of the contravariant functor $\Hom_{\Rpseudo}(-,\Rpseudo/J)$, we get 
\newline
$\Hom_{\Rpseudo}(\A_{i,j}^{\FL}/\A_{i,j}',\Rpseudo/J) = \Hom_{\Rpseudo}((\A_{i,j}^{\FL}/\A_{i,j}') \otimes_{\Rpseudo} \Rpseudo/J,\Rpseudo/J)$ 
\newline
as $\Hom_{\Rpseudo}(J(\A_{i,j}^{\FL}/\A_{i,j}'),\Rpseudo/J) = 0$. 


To simplify notation, let $E^{\red} = E^{\FL}_{\D} \otimes_{\Rpseudo} \Rpseudo/J$.  
$E^{\red}$ has GMA structure
$$
E^{\red} \cong
\begin{pmatrix}
    M_{1,1}(\Rpseudo/J) & M_{1,2}(\A_{1,2}^{\FL,\red}) & M_{1,3}(\A_{1,3}^{\FL,\red})\\
    M_{2,1}(\A_{2,1}^{\FL,\red}) & M_{2,2}(\Rpseudo/J) & M_{2,3}(\A_{2,3}^{\FL,\red})\\
    M_{3,1}(\A_{3,1}^{\FL,\red}) & M_{3,2}(\A_{3,2}^{\FL,\red}) & M_{3,3}(\Rpseudo/J)
\end{pmatrix}
$$
with induced idempotents data $\E^{\red} = \{e_i^{\red}, \psi_i^{\red}\}$ of $\E^{\sigma}$ in $E^{\red}$ i.e. $e_i^{\red}$ is the image of $e^{\FL}_{\D,i}$ in $E^{\red}$ and $\psi^{\red}_i: e_i^{\red} E^{\red} e_i^{\red} \to M_{i,i}(\Rpseudo/J)$ is an isomorphism with $\psi^{\red}_i \otimes k = \rho_i$.
Let $\A_{i,j}^{\red,'} = \sum_{k \ne i,j} \A_{i,k}^{\FL,\red} \A_{k,j}^{\FL,\red}$.

Note that as $M_{i,j}(\A_{i,j}^{\FL,\red}) \cong e_i^{\red} (E^{\FL} \otimes \Rpseudo/J) e_j^{\red} \cong M_{i,j}(\A_{i,j}^{\FL}) \otimes_{\Rpseudo} \Rpseudo/J \cong M_{i,j}(\A_{i,j}^{\FL} \otimes_{\Rpseudo} \Rpseudo/J)$, by Morita equivalence we get $\A_{i,j}^{\FL} \otimes_{\Rpseudo} \Rpseudo/J \cong \A_{i,j}^{\FL,\red}$. 

As $(\A_{i,j}^{\FL}/\A_{i,j}^{\FL,'}) \otimes_{\Rpseudo} \Rpseudo/J \cong (\A_{i,j}^{\FL} \otimes_{\Rpseudo} \Rpseudo/J)/(\A_{i,j}^{\FL,'} \otimes_{\Rpseudo} \Rpseudo/J) \cong \A_{i,j}^{\FL,\red}/((\A_{i,k}^{\FL} \otimes_{\Rpseudo} \A_{k,j}^{\FL}) \otimes_{\Rpseudo} \Rpseudo/J) \cong \A_{i,j}^{\FL,\red}/(\A_{i,k}^{\FL,\red} \A_{k,j}^{\FL,\red}) \cong \A_{i,j}^{\FL,\red}/\A_{i,j}^{\red,'}$, it suffices to show that 
\begin{equation}
    \Hom_{\Rpseudo}(\A_{i,j}^{\FL,\red}/\A_{i,j}^{\red,'},\Rpseudo/J) \cong \Ext^{1}_{\Rpseudo/J[\G], \FL} (\rho_{j,\FL}, \rho_{i,\FL})
\end{equation}

Recall that $M_{i,i}(\Rpseudo/J)$ has dimension $n_i \times n_i$. Following notation in Theorem 1.5.6 in \cite{bellaiche2009families}, let $E_i \in e_i^{\red} E^{\red} e_i^{\red} $ be the unique element such that $\psi_i^{\red}(e_i^{\red} E^{\red} e_i^{\red})$ is the elementary matrix of $M_{d_i}(\Rpseudo/J)$ with unique nonzero coefficient at row 1 and column 1. Let 
$$
M_j = E^{\red} E_j = \bigoplus_{i=1}^{3} e_i^{\red} E^{\red} E_j, \qquad N_j = \bigoplus_{i \ne j} e_i^{\red} E^{\red} E_j,
$$
then by the proof of Theorem 1.5.6 in \cite{bellaiche2009families}, we have an isomorphism
\begin{equation}{\label{eq3.12}}
    \Hom_{\Rpseudo}(\A_{i,j}^{\FL,\red}/\A_{i,j}^{\red,'},\Rpseudo/J) \cong \Ext^{1}_{E^{\red}} (\rho_{j,\FL}, \rho_{i,\FL})
\end{equation}
via a map $\delta$ such that for $f \in \Hom_{\Rpseudo}(\A_{i,j}^{\FL,\red}/\A_{i,j}^{\red,'},\Rpseudo/J) \cong \Hom_{E^{\red}}(N_j, \rho_{i,\FL})$, $\delta(f)$ is the following short exact sequence of $E^{\red}$-module:
\begin{equation}\label{ext seq}
\begin{tikzcd}
        \quad & 0 \arrow[r] &
        N_{j} \arrow[r] \arrow[d, "f"] &
        M_{j} \arrow[r] \arrow[d] &
        \rho_{j,\FL} \arrow[r] \arrow[d] &
        0
        \\
        \delta(f): & 0 \arrow[r] &
        \rho_{i,\FL} \arrow[r] &
        \frac{M_j \oplus \rho_{i,\FL}}{Q} \arrow[r] &
        \rho_{j,\FL} \arrow[r] & 
        0
    \end{tikzcd}
\end{equation}
where $Q$ is the image of $N_j$ via $E^{\red}$-linear map $u: N_j \to M_j \oplus \rho_{i,\FL}$, and $\frac{M_j \oplus \rho_{i,\FL}}{Q} \cong e_j^{\red} E^{\red} E_j \oplus e_i^{\red} E^{\red}  E_i$ as $\Rpseudo/J$-modules.

\eqref{eq3.12} gives that $\Hom_{\Rpseudo}(\A_{i,j}^{\FL,\red}/\A_{i,j}^{\red,'},\Rpseudo/J) \hookrightarrow \Ext^{1}_{\Rpseudo/J[G_{\Sigma}]} (\rho_{j,\FL}, \rho_{i,\FL})$ is an injection. It remains to show that $\delta(f)$ is a Fontaine--Laffaille at $p$ extension to show such extensions are in $\Ext^{1}_{\Rpseudo/J[G_{\Sigma}], \FL} (\rho_{j,\FL}, \rho_{i,\FL})$. We need to show that $\frac{M_j \oplus \rho_{i,\FL}}{Q}$ is Fontaine--Laffaille at $p$ as $\Rpseudo/J[G_{\Sigma}]$ module. Recall that the Fontaine--Laffaille at $p$ condition is closed under subquotients and direct sums, therefore it suffices to prove that $M_j$ is Fontaine--Laffaille at $p$. By Theorem \ref{N has C}, $M_j$ is a Fontaine--Laffaille extension if $E^{\red}$ is a Fontaine--Laffaille at $p$ as an $\Rpseudo/J[G_{\Sigma}]$-module, which follows from that $E^{\FL}_{\D}$ is Fontaine--Laffaille at $p$.

We need to show that \eqref{hom-ext iso} is a surjection. 
Suppose $c \in \Ext^{1}_{\Rpseudo/J[G_{\Sigma}], \FL} (\rho_{j,\FL}, \rho_{i,\FL})$, then
$$
c: \qquad 0 \to \rho_{i,\FL} \to E_{i,j} \to \rho_{j,\FL} \to 0 
$$
for some $E_{i,j}$. 
We can construct a Fontaine--Laffaille at $p$ $\Rpseudo/J[\G]$-module $\title{E}_c = E_{i,j} \oplus \rho_{k,FL}$ where $k \ne i,j$ from $E_{i,j}$, where the natural $\G$ action factors through $E^{\red}$, and a reducible Fontaine Laffaille at $p$ pseudodeformation of $\D$, $(\title{E}_c, \rho_c, D_c: \Rpseudo/J[G_{\Sigma}] \to \Rpseudo/J)$. Choose a GMA structure $\{e_i^{c}, \psi_i^{c}\}$ for $\title{E}_c$ such that $e_i^{c} \title{E}_c e_i^{c} = \rho_{i,\FL}$ for $i = 1,2,3$. 

By universality of $(E_{\D}^{\FL},\rho_{\D}^{\FL},D_{\D}^{\FL})$ and $D_c$ being totally reducible, we obtain a morphism of GMAs: $E_{\D}^{\FL} \otimes (\Rpseudo/I_{R_{\D}^{\FL}}^{\tot}) \to E_c$ where $I_{R_{\D}^{\FL}}^{\tot}$ is the total reducibility ideal of $R_{\D}^{\FL}$. This morphism then factors through $E^{\red}$ as $I_{R_{\D}^{\FL}}^{\tot} \subset J$ and $R/I_{R_{\D}^{\FL}}^{\tot} \otimes R/J \cong R/J$, and we obtain a morphism of GMAs $g: E^{\red} \to E_c$ such that $g(e_i^{\red}) = e_i^{c}$ for all $i$ (cf. Theorem 3.2.2 in \cite{wake2019deformation}). Then by Morita equivalence, $g|_{e_i^{\red} E^{\red} e_j^{\red}}$ as a $\Rpseudo$-module homomorphism induces an element
$\Tilde{g} \in \Hom_{\Rpseudo}(\A_{i,j}^{\FL,\red},\Rpseudo/J)$. As $g(e_i^{\red} E^{\red} e_k^{\red}) = 0$, $\Tilde{g}$ is trivial on $\A_{i,j}^{\red,'}$ and $\Tilde{g} \in \Hom_{\Rpseudo}(\A_{i,j}^{\FL,\red}/\A_{i,j}^{\red,'},\Rpseudo/J)$.  And $c = \delta(\Tilde{g})$ by such construction. Thus $\delta$ is a surjection.
\end{proof}

\begin{prop}\label{Ext with C}
Let $J \subset R$ be an ideal such that $I^{tot} \subset J$. Let $A_{i,j}' = \sum_{k \ne i,j} A_{i,k}^{\sigma} A_{k,j}^{\sigma}$. Then there is an injection
\begin{equation}{\label{ext eq for S}}
    \Hom_{R}(A_{i,j}^{\sigma}/A_{i,j}',R/J) \hookrightarrow \Ext^{1}_{R/J[G_{\Sigma}], \FL} (\rho_{j,\FL}, \rho_{i,\FL})
\end{equation} 
\end{prop}
\begin{proof}
This follows a similar argument as in Proposition \ref{Ext with C pseudo} without showing \eqref{ext eq for S} is a surjection.
\end{proof}

\subsubsection{\textsf{Principality the Reducibility Ideal}}
In this section, we use Assumptions \ref{main assumptions} (1)(2)(3) to show principality of the reducibility ideal.

\begin{prop}
For $\star \in \{\sigma,D\}$, using Assumption \ref{main assumptions} (1), we have 
$$
A_{1,3}^{\star} = A_{1,2}^{\star} A_{2,3}^{\star}, A_{3,1}^{\star} = A_{3,2}^{\star} A_{2,1}^{\star}.
$$
and
$$
A_{1,3}^{\star} A_{3,1}^{\star} \subseteq A_{1,2}^{\star} A_{2,1}^{\star}.
$$
And it follows from the existence of the anti-involution that
$$
I^{\tot} = A_{1,2}^{D} A_{2,1}^{D}.
$$
\end{prop}
\begin{proof}
In Proposition \ref{Ext with C}, we can take $J = \mathfrak{m}$, and then we get
$$
    \Hom_R (A_{i,j}^{\sigma}/A_{i,j}^{'},k) \hookrightarrow H^1_{\Sigma}(F,\Hom(\rho_j, \rho_i)).
$$
Since by Assumption \ref{main assumptions} (1), $H^1_{\Sigma}(F,\Hom(\rho_3, \rho_1)) = 0$, we get $A_{1,3}^{\sigma} = A_{1,2}^{\sigma} A_{2,3}^{\sigma}$ by complete version of Nakayama's lemma. Let $E_i^{\star} \in e_i^{\star} E^{\star} e_i^{\star}$ be the unique element such that $\psi_i^{\star}(e_i^{\star} E^{\star} e_i^{\star})$ is the matrix unit of $M_{d_{i}}(R)$ with unique nonzero coefficient 1 at row 1 and column 1. Then $\phi (A_{1,2}^{\sigma} A_{2,3}^{\sigma}) = \phi (A_{1,2}^{\sigma}) \phi(A_{2,3}^{\sigma})$ by identifying $A_{i,j}^{\sigma} = E_i^{\sigma} E^{\sigma} E^{\sigma}_j$ and $\phi (E_i^{\sigma}) = E_i^{D}$ (cf. \cite[Lemma 2.5]{berger2013deformation}). And it follows from Lemma \ref{GMA}(2) that $A_{1,3}^{D} = A_{1,2}^{D} A_{2,3}^{D}$.

Then as $A_{2,3}^{\star} A_{3,1}^{\star} \subseteq A_{2,1}^{\star}$, 
$$
A_{1,3}^{\star} A_{3,1}^{\star} = A_{1,2}^{\star} A_{2,3}^{\star} A_{3,1}^{\star} \subseteq A_{1,2}^{\star} A_{2,1}^{\star}.
$$
Also, we have $A_{1,2}^{D} A_{2,1}^{D} = A_{2,3}^{D} A_{3,2}^{D}$ as a result of the existence of anti-involution (cf. Lemma 1.8.5 (ii) in \cite{bellaiche2009families}). Thus by Proposition \ref{red ideal structure},
\begin{align*}
    I^{\tot} &= A_{1,3}^{D} A_{3,1}^{D} + A_{1,2}^{D} A_{2,1}^{D} + A_{2,3}^{D} A_{3,2}^{D}\\
    &= A_{1,2}^{D} A_{2,1}^{D}.
\end{align*}
\end{proof}

\begin{prop}\label{principality}
For $\star \in \{\sigma,D\}$, we have $A_{1,2}^{\star}, A_{2,1}^{\star}$ are free cyclic $R$-modules and $I^{\tot}$ is principal.
\end{prop}
\begin{proof}
Note $A_{2,1}^{'} = A_{2,3}^{\sigma} A_{3,1}^{\sigma} = A_{2,3}^{\sigma} A_{3,2}^{\sigma} A_{2,1}^{\sigma} \subset \mathfrak{m} A_{2,1}^{\sigma}$. Similarly,  $A_{1,2}^{'} = A_{1,3}^{\sigma} A_{3,2}^{\sigma} = A_{1,2}^{\sigma} A_{2,3}^{\sigma} A_{3,2}^{\sigma} \subset \mathfrak{m} A_{1,2}^{\sigma}$. Thus by Assumptions (2)(3) $\dim_k H^{1}_{\Sigma} (F, \Hom(\rho_1, \rho_2)) = \dim_k H^{1}_{\Sigma} (F, \Hom(\rho_2, \rho_1)) = 1$ and Proposition \ref{Ext with C}, we get 
$$
\Hom(A_{2,1}^{\sigma}/\mathfrak{m}A_{2,1}^{\sigma},k) \hookrightarrow \Hom(A_{2,1}^{\sigma}/A_{2,1}^{'},k) \hookrightarrow k,
$$
and 
$$
\Hom(A_{1,2}^{\sigma}/\mathfrak{m}A_{2,1}^{\sigma},k) \hookrightarrow \Hom(A_{1,2}^{\sigma}/A_{1,2}^{'},k) \hookrightarrow k.
$$
Then $A_{2,1}^{\sigma}/\mathfrak{m}A_{2,1}^{\sigma}, A_{1,2}^{\sigma}/\mathfrak{m}A_{1,2}^{\sigma}$ is generated by one element by the complete version Nakayama's lemma as it is non-trivial. Applying $\phi$ to $A_{2,1}^{\sigma}, A_{1,2}^{\sigma}$, we get $A_{2,1}^{D}, A_{1,2}^{D}$ are principal ideal of $R$ and $I^{\tot}$ is principal.
\end{proof}

\subsubsection{\textsf{$\ker D = \ker \sigma$}} Recall the definition of $\Tilde{\rho_i}$ in Assumption \ref{main assumptions} (4). In this part, we show that under Assumptions \ref{main assumptions} (2)(3) and $\# H^1_{\Sigma}(F,\Hom(\Tilde{\rho_3},\Tilde{\rho_2}) \otimes_{E} E/\mO) \le \# \mO / \varpi^r \mO$ for some $r$, $\ker D = \ker \sigma$. The same bound can be imposed on $\# H^1_{\Sigma}(F,\Hom(\Tilde{\rho_2},\Tilde{\rho_1}) \otimes_{E} E/\mO)$ due to the existence of the anti-involution $\tau$. This part can be viewed independently as the arguments in the rest of the section can proceed without it.

\begin{lem}
Using Assumptions \ref{main assumptions} (2)(3), if $R$ is infinite and $\# H^1_{\Sigma}(F,\Hom(\Tilde{\rho_3},\Tilde{\rho_2}) \otimes_{E} E/\mO) \le \# \mO / \varpi^r \mO$ for some $r$, then we get $\ker D = \ker \sigma$ and $A_{i,j}^{\sigma}$ are ideals of $R$.
\end{lem}
\begin{proof}
By a similar argument as in Proposition 3.1 in \cite{berger2013deformation}, which is not specific to a particular number of Jordan-H\"older factors, to show $\ker D = \ker \sigma$, we need to show that $R$ is reduced, infinite and $\# R/I^{\tot}$ is finite. $R$ is reduced by definition. $R$ is infinite by assumption (and can be verified if $\overline{\sigma}$ admits a deformation to $\GL_n(\mO)$ that is Fontaine--Laffaille and self-dual). It suffices to show that $\# R/I^{\tot}$ is finite, which follows from a similar argument as in Corollary \ref{No deformations to Z/Z_p^2} using $\# H^1_{\Sigma}(F,\Hom(\Tilde{\rho_3},\Tilde{\rho_2}) \otimes_{E} E/\mO) \le \# \mO / \varpi^r \mO$.

Then it follows that $A_{i,j}^{\sigma}$ are also fractional ideals of $R$ from Lemma \ref{GMA}, as they are also submodules of $R$, $A_{i,j}^{\sigma}$ are ideals of $R$.
\end{proof}

\section{Reducible Deformations and $R/I^{\tot}$}{\label{R is DVR theorem section}}
This section studies the reducible deformations of $\overline{\sigma}$ as in Section \ref{deformation problem}. The goal is to show that 
\begin{enumerate}
    \item $\dim_{k} H^1_{\Sigma}(F, \Hom(\rho_2,\rho_1))$ and $\dim_{k} H^1_{\Sigma}(F, \Hom(\rho_3,\rho_2))$ 
    control the equal characteristic reducible upper-triangular deformations;
    \item $H^1_{\Sigma}(F,\Hom(\Tilde{\rho_2},\Tilde{\rho_1}) \otimes E/\mO)$, $H^1_{\Sigma}(F,\Hom(\Tilde{\rho_3},\Tilde{\rho_2}) \otimes E/\mO)$ 
    control the size of upper triangular deformations to $\GL_{n}(\mO/\varpi^{s} \mO)$, where $\Tilde{\rho_i}$ is defined in Assumption \ref{main assumptions} (4), for $s \ge 1$,
\end{enumerate}
and we can control the size of $R/I^{\tot}$ by the above Selmer groups.

\begin{defn}
A deformation $\sigma$ is \textbf{upper-triangular} if some element in its strict equivalence class is of the form
$$
\sigma(g) = 
\begin{pmatrix}
    A_1(g) & * & * \\
    0 & A_2(g) & * \\
    0 & 0 & A_3(g)
\end{pmatrix}
\text{ for every } g \in \G,
$$
where $A_i(g)$ are $n_i \times n_i$ matrices, for $i = 1,2,3$.
\end{defn}

\begin{prop}{\label{No deformations to f[x]/x^2}}
    There is no non-trivial upper-triangular deformation of $\overline{\sigma}$ to $k[X]/X^2$ that is Fontaine--Laffaille at $p$.
\end{prop}
\begin{proof}
The argument of Proposition 7.3 in \cite{berger2020deformations} could apply after generalizing matrix algebra dimensions and replacing the assumptions on the Bloch--Kato Selmer groups $H^1_f$ by assumptions on the relaxed Selmer groups $H^1_{\Sigma}$.

First, note that because of the assumption that there is only one equivalence class of deformation of $\rho_i$ that is Fontaine-Laffaile at $p$ (Assumption \ref{main assumptions} (4)), 
if $\sigma': \G \to \GL_n(k[X]/X^2)$ is such a deformation, then the diagonal entries of $\sigma'$ do not deform and $\sigma'$ is of the form:
\begin{equation*}
    \sigma' = 
    \begin{pmatrix}
        \rho_1 & a + X a' &  b+ X b' \\
        0 & \rho_2 & c + X c' \\
        0 & 0 & \rho_3 \\
    \end{pmatrix}.
\end{equation*}
As 
\begin{equation*}
    \begin{pmatrix}
        \rho_2 & c + X c'\\
        0 & \rho_3\\
    \end{pmatrix}
\end{equation*}
is a quotient of $\sigma'$, it is Fontaine--Laffaille at $p$. Thus arguing as in Proposition 7.2 in \cite{berger2013deformation} using $\dim_{k} H^1_{\Sigma}(F, \Hom(\rho_3,\rho_2)) = 1$ (Assumption \ref{main assumptions} (2)), we can conjugate
$
\begin{pmatrix}
    \rho_2 & c + X c'\\
    0 & \rho_3\\
\end{pmatrix}
$
to
$
\begin{pmatrix}
    \rho_2 & c\\
    0 & \rho_3\\
\end{pmatrix}
$
by some $M_1 \in I_{r} + X M_{r \times r}(k[X]/X^2)$, where $r = n_2 + n_3$ and $I_{r}$ is the identity matrix of dimension $r \times r$.

Then
\begin{equation}{\label{eq3}}
\sigma' = 
\begin{pmatrix}
    \rho_1 & a + X a' &  b+ X b' \\
    0 & \rho_2 & c + X c' \\
    0 & 0 & \rho_3 \\
\end{pmatrix}
\cong
\begin{pmatrix}
    \rho_1 & a + X a'' &  b+ X b'' \\
    0 & \rho_2 & c \\
    0 & 0 & \rho_3 \\
\end{pmatrix},
\end{equation}
if we conjugate $\sigma'$ by
$
\begin{pmatrix}
    I_{n_1 \times n_1} & 0\\
    0 & M_1
\end{pmatrix}
$. Note the image of $\sigma'$ in $\GL_{n}(k)$ remains the same after conjugation as $M_1$'s image in $\GL_{n}(k)$ is the identity.


Similarly, we use $\dim_{k} H^1_{\Sigma}(F, \Hom(\rho_2,\rho_1)) = 1$ and $H^1_{\Sigma}(F, \Hom(\rho_3,\rho_1)) = 0$ to show $a'' = 0$ and $b'' = 0$ respectively. 
With respect to the basis
\begin{equation}{\label{basis}}
    \Bigg\{
\begin{pmatrix}
    1 \\
    0 \\
    \vdots \\
    0 \\
    0 
\end{pmatrix}
,
\begin{pmatrix}
    0 \\
    1 \\
    \vdots \\
    0 \\
    0 
\end{pmatrix}
,
...
,
\begin{pmatrix}
    0 \\
    0 \\
    \vdots \\
    0 \\
    1 
\end{pmatrix}
,
\begin{pmatrix}
    X \\
    0 \\
    \vdots \\
    0 \\
    0 
\end{pmatrix}
,
\begin{pmatrix}
    0 \\
    X\\
    \vdots \\
    0 \\
    0 
\end{pmatrix}
,
...
,
\begin{pmatrix}
    0 \\
    0 \\
    \vdots \\
    0 \\
    X
\end{pmatrix}
\Bigg \},
\end{equation}
$\sigma'$ has the form
\begin{equation}\label{eq2}
    \begin{pmatrix} 
    \rho_1 & a & b & 0 & 0 & 0\\
    0 & \rho_2 & c & 0 & 0 & 0 \\
    0 & 0 & \rho_3 & 0 & 0 & 0 \\
    0 & a'' & b'' & \rho_1 & a & b \\
    0 & 0 & 0 & 0 & \rho_2 & c \\
    0 & 0 & 0 & 0 & 0 & \rho_3 \\
\end{pmatrix}.
\end{equation}
As a quotient of a sub-representation of \eqref{eq2},
$$
\begin{pmatrix}
    \rho_2 & c & 0 \\
    0 & \rho_3 & 0 \\
    a'' & b'' & \rho_1
\end{pmatrix}
\cong
\begin{pmatrix}
    \rho_2 & 0 & c \\
    a'' & \rho_1 & b'' \\
    0 & 0 & \rho_3
\end{pmatrix}
$$
is Fontaine--Laffaille at $p$. Thus as its submodule
$$
\begin{pmatrix}
    \rho_2 & 0 \\
    a'' & \rho_1
\end{pmatrix}
$$
gives rise to a non-zero element in $H^1_{\Sigma}(F, \Hom(\rho_2,\rho_1))$, which has dimension 1, and $a'' = \alpha a$ for some scalar $\alpha \in k^{\times}$.

Then \eqref{eq3} becomes
$$
\sigma'
\cong
\begin{pmatrix}
    \rho_1 & a + X \alpha a &  b + X b'' \\
    0 & \rho_2 & c\\
    0 & 0 & \rho_3 \\
\end{pmatrix}
\cong 
\begin{pmatrix}
    \rho_1 & a &  b + X b''' \\
    0 & \rho_2 & c\\
    0 & 0 & \rho_3 \\
\end{pmatrix},
$$
where the latter isomorphism follows by conjugation of $\sigma$ by some $
\begin{pmatrix}
    M_2 & 0\\
    0 & I_{n_3 \times n_3}
\end{pmatrix},
$
for some $M_2 \in I_{s \times s} + X M_{s \times s}(k[X]/X^2)$ and $s = n_1 + n_2$.

It remains to show that $b+ X b''' \in k$. With respect to the same basis \eqref{basis}, there is a submodule of a quotient of $\sigma'$ of the form
$$
\begin{pmatrix}
    \rho_3 & 0 \\
    b''' & \rho_1
\end{pmatrix},
$$
giving rise to an element in $H^1_{\Sigma}(F, \Hom(\rho_3,\rho_1)) = 0$. Thus we can assume $b''' =0$ by a similar argument to the one above, so 
$$
\sigma'
\cong 
\begin{pmatrix}
    \rho_1 & a &  b\\
    0 & \rho_2 & c\\
    0 & 0 & \rho_3 \\
\end{pmatrix}.
$$
as desired.
\end{proof}

\begin{prop}\label{upper triangular and I tot}
Suppose $J \subset R_{\FL}$ is an ideal such that $R_{\FL}/J$ is an object of $\hat{\mathcal{C}}_{W(k)}$. Then $I_{R_{\FL}}^{\tot} \subseteq J$ if and only $\sigma_{\FL} \otimes_{R_{\FL}} R_{\FL}/J$ is an upper-triangular deformation. 
\end{prop}
\begin{proof}
This can be proved similarly to Lemma 7.5 in \cite{berger2020deformations} using the definition of reducibility ideal and Theorem 1.1 of \cite{brown2008residually}.
\end{proof}

\begin{prop}{\label{no deformation to Z/ps}}
Suppose $\# H^1_{\Sigma}(F,\Hom(\Tilde{\rho_2},\Tilde{\rho_1}) \otimes_{E} E/\mO) \le \# \mO / \varpi^r \mO$ or 
\newline
$\# H^1_{\Sigma}(F,\Hom(\Tilde{\rho_3},\Tilde{\rho_2}) \otimes_{E} E/\mO) \le \# \mO / \varpi^r \mO$. Then there is no non-trivial upper-triangular deformation of $\overline{\sigma}$ to $\mO/\varpi^{s} \mO$ that is Fontaine--Laffaille at $p$ for $s > r$. And $\# R_{\FL}/I_{R_{\FL}}^{\tot} \le \# \mO/\varpi^{r} \mO$. 
\end{prop}
\begin{proof}
We assume that $\# H^1_{\Sigma}(F,\Hom(\Tilde{\rho_2},\Tilde{\rho_1}) \otimes_{E} E/\mO) \le \# \mO / \varpi^r \mO$. The case where $\# H^1_{\Sigma}(F,\Hom(\Tilde{\rho_3},\Tilde{\rho_2}) \otimes_{E} E/\mO) \le \# \mO / \varpi^r \mO$ is analogous due to the existence of the anti-involution. Suppose there exists such a deformation $\sigma': \G \to \GL_n(\mO/\varpi^{s} \mO)$. By Assumption \ref{main assumptions} (4), we know
\begin{equation*}
    \sigma' = 
\begin{pmatrix}
    \Tilde{\rho}_{1} \otimes (\mO/\varpi^{s} \mO) & *_1 & *_2\\
    0 & \Tilde{\rho}_{2}  \otimes (\mO/\varpi^{s} \mO)  & *_3\\
    0 & 0 & \Tilde{\rho}_{3} \otimes (\mO/\varpi^{s} \mO)
\end{pmatrix}.
\end{equation*}
Then $*_1$ gives rise to elements in $H^1_{\Sigma}(F,\Hom(\Tilde{\rho_2},\Tilde{\rho_1}) \otimes E/\mO[\varpi^{s}]) = H^1_{\Sigma}(F,\Hom(\Tilde{\rho_2},\Tilde{\rho_1}) \otimes E/\mO)[\varpi^{s}]$ by Lemma \ref{torsion Selmer is Selmer torsion}. As the reduction of $\sigma'$ contains $\begin{pmatrix}
    \rho_1 & *_1 \\
    & \rho_2
\end{pmatrix}$ which is non-split, 
so $*_1$ gives rise to a non-trivial element of $H^1_{\Sigma}(F,\Hom(\Tilde{\rho_2},\Tilde{\rho_1}) \otimes_{E} E/\mO)[\varpi]$, and the image of $*_1$ is not contained in $\varpi \mathcal{O}/\varpi^s \mathcal{O}$. This implies that $*_1$ generates a submodule of $H^1_{\Sigma}(F,\Hom(\Tilde{\rho_2},\Tilde{\rho_1}) \otimes_{E} E/\mO)$ of order $\ge \# \mO/\varpi^s \mO$, contradicting the assumption that $\# H^1_{\Sigma}(F,\Hom(\Tilde{\rho_2},\Tilde{\rho_1}) \otimes_{E} E/\mO) \le \# \mO / \varpi^r \mO$. Thus there are no non-trivial upper-triangular deformations of $\overline{\sigma}$ to $\mO/\varpi^{s} \mO$ that is Fontaine--Laffaille at $p$ for $s > r$.

Suppose $\# R_{\FL}/I_{R_{\FL}}^{\tot} = \# \mO/\varpi^{s} \mO$ for some $s > r$. Then by Proposition \ref{upper triangular and I tot}, there exists an upper-triangular deformation to $ \mO/\varpi^{s} \mO$ that is Fontaine Laffaille at $p$, but by above, no such deformation exists.
\end{proof}

\begin{cor}{\label{No deformations to Z/Z_p^2}}
If $\# H^1_{\Sigma}(F,\Hom(\Tilde{\rho_2},\Tilde{\rho_1}) \otimes_{E} E/\mO) \le \# \mO / \varpi^r \mO$ or 
\newline
If  $\# H^1_{\Sigma}(F,\Hom(\Tilde{\rho_3},\Tilde{\rho_2}) \otimes_{E} E/\mO) \le \# \mO / \varpi^r \mO$, then
$$
\# R/I^{\tot} \le \# \mO/\varpi^{r} \mO.
$$
\end{cor}
\begin{proof}
This follows from Proposition \ref{no deformation to Z/ps} above and Lemma 7.11 in \cite{berger2013deformation} as $R$ is a quotient of $R_{\FL}$.
\end{proof}


The following Theorem summarizes the results that allow us to identify when $R$ is a discrete valuation ring. For the convenience of the reader, we include all the assumptions made throughout the sections.

\begin{thm}{\label{main thoerem}}
Consider the following residual representation: $\overline{\sigma}: \G \to \GL_{n}(k)$,
\begin{equation}
\overline{\sigma} = 
\begin{pmatrix}
    \rho_1 & a & b \\ 
    0 & \rho_2 & c \\ 
    0 & 0 & \rho_3 \\ 
\end{pmatrix},
\end{equation}
where 
\begin{enumerate}
    \item $\overline{\sigma}$ is Fontaine--Laffaille at $p$;
    \item $\rho_i: G_{\Sigma} \to \GL_{n_{i}}(k)$ are absolutely irreducible for each $i$ and $\rho_i \not \cong \rho_j$ for $i \ne j$ and $R_{i,\FL} = \mO$ for all $1 \le i \le 3$;
    \item $\begin{pmatrix}
    \rho_1 & a \\
    0 & \rho_2
    \end{pmatrix}
    $
    and 
    $
    \begin{pmatrix}
        \rho_2 & c \\
        0 & \rho_3
    \end{pmatrix}
    $
    are non-split extensions;
    \item $\overline{\sigma}$ is $\tau$-self-dual for some anti-involution $\tau$.
    
    \end{enumerate}
    If both of the following conditions are satisfied:
    \begin{enumerate}
    \renewcommand{\labelenumi}{(\roman{enumi})}
        \item $H^{1}_{\Sigma}(F, \Hom(\rho_3,\rho_1)) = 0$;
        \item $\dim_{k} H^{1}_{\Sigma}(F, \Hom(\rho_1,\rho_2))= 1$;
        \item $\# H^1_{\Sigma}(F,\Hom(\Tilde{\rho_2},\Tilde{\rho_1}) \otimes_{E} E/\mO) \le \# \mO/\varpi \mO$,
    \end{enumerate}
    then the maximal ideal $\m$ of $R$ is principal. Furthermore, suppose $\overline{\sigma}$ admits a deformation to $\GL_n(\mO)$ that is Fontaine--Laffaille and self-dual, then $R \cong \mO$.

\end{thm}
\begin{proof}
This follows a similar argument to Lemma 3.4 and Lemma 3.5 in \cite{calegari2006eisenstein}. It follows from Proposition \ref{No deformations to f[x]/x^2} and Corollary \ref{No deformations to Z/Z_p^2} that there is no surjection $R/I^{\tot} \to k[X]/X^2$ or $R/I^{\tot} \to \mO/\varpi^2 \mO$, and as a result, $I^{\tot}$ is maximal. Proposition \ref{principality} now shows that the maximal ideal of $R$ is $(\varpi)$. Suppose $\overline{\sigma}$ admits a deformation to $\GL_n(\mO)$ that satisfies all deformation conditions, then there is a surjection $R \twoheadrightarrow \mO$ and $\varpi$ is not nilpotent. It follows that $R \cong \mO$.
\end{proof}

\begin{rmk}
The conditions on $\overline{\sigma}$, i.e., (1)-(3) are insufficient to guarantee the uniqueness of $\overline{\sigma}$. For example, $\begin{pmatrix}
    \rho_1 & a \\
    0 & \rho_2
    \end{pmatrix}
    $
and 
$
\begin{pmatrix}
    \rho_2 & c \\
    0 & \rho_3
\end{pmatrix}
$
might not be unique. However, with the additions of (i)(ii), $\overline{\sigma}$ is unique up to isomorphism (Proposition  \ref{Uniqueness of residual}).
\end{rmk}


\section{Lattice}

In this section, we discuss the existence of a lattice $\Lambda$ related to a deformation $\sigma'$ of $\sigma$ such that the three Jordan-H\"{o}lder factors in the associated upper-triangular reduction are ordered exactly as $\rho_1, \rho_2, \rho_3$ and $\begin{pmatrix}
    \rho_1 & a\\
    0 & \rho_2
\end{pmatrix},
\begin{pmatrix}
    \rho_2 & c\\
    0 & \rho_3
\end{pmatrix}$
are non-split extensions. This section can be viewed independently, without the running assumptions stated in Assumption \ref{main assumptions} or the conditions on $\overline{\sigma}$ in Section \ref{general residual set up}.

\begin{prop}{\label{Lattice}}
Let $\sigma': G_{\Sigma} \to \GL_{n}(E)$ be absolutely irreducible,  $\overline{\sigma^{'}}^{\sss} \cong \rho_1 \oplus \rho_2 \oplus \rho_3$ and is Fontaine--Laffaille at $p$, where $\rho_1, \rho_2, \rho_3$ are pairwise non-isomorphic and irreducible. Assume Assumption \ref{main assumptions} (1). Then there exists a $G_{\Sigma}$-stable lattice $\Lambda$ such that the associated reduction $\overline{\sigma'}_{\Lambda}$ satisfies
$$
\overline{\sigma'}_{\Lambda} \cong 
    \begin{pmatrix}
        \rho_1 & a & b \\
        0  & \rho_2 & c \\
        0  & 0 & \rho_3 \\
    \end{pmatrix}
$$
where $
\begin{pmatrix}
    \rho_1 & a\\
    0 & \rho_2
\end{pmatrix},
\begin{pmatrix}
    \rho_2 & c\\
    0 & \rho_3
\end{pmatrix}$
are non-split extensions.
\end{prop}

\begin{proof}
We follow a similar proof as the one in \cite[Corollary 4.3]{berger2020deformations}. Theorem 6.1 in \cite{brown2011cuspidality} implies that there is a lattice $\Lambda$ such that the associated reduction $\overline{\sigma'}_{\Lambda}$ satisfies
\begin{equation}{\label{eq 5.2}}
    \overline{\sigma'}_{\Lambda} = 
    \begin{pmatrix}
        \rho_1 & 0 & b \\
        0  & \rho_2 & c \\
        0  & 0 & \rho_3 \\
    \end{pmatrix}
    \not \cong \rho_1 \oplus \rho_2 \oplus \rho_3.
\end{equation}

First, we will show that $c$ gives rise to a non-split extension of $\rho_3$ by $\rho_2$. Indeed, if $c$ gives a split extension of $\rho_3$ by $\rho_2$, then
$$
\overline{\sigma'}_{\Lambda} \cong
\begin{pmatrix}
        \rho_1 & 0 & b \\
        0  & \rho_2 & 0 \\
        0  & 0 & \rho_3 \\
 \end{pmatrix},
$$
and has a direct summand
\begin{equation} \label{direct summand}
    \begin{pmatrix}
    \rho_1 & b \\
    0 & \rho_3 \\
 \end{pmatrix},
\end{equation}
which is not a split extension by \eqref{eq 5.2}, giving rise to a non-zero element in $H^{1}_{\Sigma} (F, \Hom(\rho_3, \rho_1))$ that is 0. Thus $c$ gives rise to a non-split extension of $\rho_3$ by $\rho_2$.

$\begin{pmatrix}
    \rho_2 & c \\
    0 & \rho_3
\end{pmatrix}$ 
has a scalar centralizer. Applying Theorem 4.1 in \cite{berger2020deformations}, we have 
$$
\overline{\sigma'}_{\Lambda} = 
    \begin{pmatrix}
        \rho_1 & a & b \\
        0  & \rho_2 & c \\
        0  & 0 & \rho_3 \\
\end{pmatrix}
\not \cong \rho_1 \oplus \begin{pmatrix}
\rho_2 & c \\
  & \rho_3
\end{pmatrix}.
$$
It remains to show that $a$ gives rise to a non-split extension of $\rho_2$ by $\rho_1$, i.e.,
$$
\begin{pmatrix}
\rho_1 & a \\
0 & \rho_2
\end{pmatrix}
\not \cong \rho_1 \oplus \rho_2.
$$

A conjugation computation shows that if $a$ is a coboundary in $H^{1}(\Q,\Hom(\rho_2,\rho_1))$, then 
$$
\begin{pmatrix}
    \rho_1 & a & b \\
    0  & \rho_2 & c \\
    0  & 0 & \rho_3 \\
\end{pmatrix}
\cong 
\begin{pmatrix}
    \rho_1 & 0 & b' \\
    0  & \rho_2 & c \\
    0  & 0 & \rho_3 \\
\end{pmatrix}
\cong 
\begin{pmatrix}
    \rho_2 & 0 & c \\
    0  & \rho_1 & b' \\
    0  & 0 & \rho_3 \\
\end{pmatrix}.
$$
As a quotient of $\overline{\sigma'}$, 
$
\begin{pmatrix}
    \rho_1& b' \\
    0 & \rho_3 \\
\end{pmatrix}
$
is Fontaine--Laffaille at $p$, thus giving rise to an element in $H^{1}_{\Sigma} (F, \Hom(\rho_3, \rho_1)) = 0$. Thus $b'$ must be a coboundary. Also, we have by conjugation,
$$
\begin{pmatrix}
    \rho_2 & 0 & c \\
    0  & \rho_1 & b' \\
    0  & 0 & \rho_3 \\
\end{pmatrix}
\cong
\begin{pmatrix}
    \rho_2 & 0 & c \\
    0  & \rho_1 & 0 \\
    0  & 0 & \rho_3 \\
\end{pmatrix}
$$
contradicting 
\begin{equation*}
    \overline{\sigma'}_{\Lambda}
    \not \cong \rho_1 \oplus 
    \begin{pmatrix}
    \rho_2 & c \\
    0 & \rho_3
    \end{pmatrix}.
\end{equation*}
Thus $a$ must give rise to a non-split extension of $\rho_2$ by $\rho_1$.
\end{proof}

\begin{rmk}{\label{Ribets lemma}}
(A discussion on Ribet's Lemma) 
In Proposition \ref{Lattice}, we obtained a lattice such that the three Jordan-H\"{o}lder factors in the associated reduction are ordered exactly as $\rho_1, \rho_2, \rho_3$ because we utilized the Fontaine--Laffaille at $p$ condition. However, we note the argument also holds with the Fontaine--Laffaille at $p$ condition replaced by any other condition that is preserved when taking subquotients. The arithmetic aspect of the Fontaine--Laffaille at $p$ condition was not needed for the Proposition.

Generally, using the same argument and an analogous Selmer group condition, for any arbitrary ordering of $\{\rho_1,\rho_2,\rho_3\}$, say $\rho_i,\rho_j,\rho_k$,  we can find a lattice such that $\begin{pmatrix}
    \rho_i & *\\
    0 & \rho_j
\end{pmatrix},
\begin{pmatrix}
    \rho_j & *\\
    0 & \rho_k
\end{pmatrix}$
are non-split extensions, extending Ribet's Lemma \cite[Proposition 2.1]{ribet1976modular} to residual representations with 3 Jordan-H\"{o}lder factors. However, without such a condition, we could only achieve a weaker conclusion. 

In a general setting, let $\R \in \text{ob}(\hat{\C}_{W(k)})$ be a domain with field of fractions $\text{Tot}(\R)$, and $\rho: \G \to \GL_m(\text{Tot}(\R))$ be an absolutely irreducbile representation of $G_{\Sigma}$ on $\text{Tot}(\R)^m$. Let $\overline{\rho}$ be the mod-$p$ representation of $\rho$ on $k^m$ with $s$ Jordan-H\"{o}lder factors. Suppose $\overline{\rho}_1$, ..., $\overline{\rho}_{s}$ are the Jordan-H\"{o}lder factors of $\overline{\rho}$ such that $\overline{\rho}_i$'s are absolutely irreducible and $\overline{\rho}_i \not \cong \overline{\rho}_j$ for $i \ne j$. 

If $\Gamma$ is a graph whose vertices are $\{\overline{\rho}_i\}$, and there is a directed edge from $\overline{\rho}_i$ to $\overline{\rho}_j$ if there exists an non-split extension of $\overline{\rho}_j$ by $\overline{\rho}_i$, then Corollary 1 in \cite{Bellaïche2003} gives a generalized Ribet's Lemma: $\Gamma$ is connected as a directed graph. Following this, the proposition below distinguishes between the case with two Jordan-H\"{o}lder factors and the case with a higher number of Jordan-H\"{o}lder factors. We summarize Corollary 1 in \cite{Bellaïche2003} below.
\end{rmk}

\begin{prop}
Suppose $\Bar{\rho}$ has only two Jordan-H\"{o}lder factors, i.e., $s = 2$, then there exist respective lattices such that the associated reductions of $\rho$, $\begin{pmatrix}
    \overline{\rho}_1 & *\\
    & \overline{\rho}_2
\end{pmatrix}$, $\begin{pmatrix}
    \overline{\rho}_2 & *\\
    & \overline{\rho}_1
\end{pmatrix}$
are non-split from either direction, generalizing Ribet's Lemma for residual representations with 2 Jordan-H\"{o}lder factors.

Suppose $\Bar{\rho}$ has more than two Jordan-H\"{o}lder factors, i.e., $s > 2$, one can only conclude that $\Gamma$ is connected as a directed graph, without concluding which extensions $\begin{pmatrix}
    \overline{\rho}_i & *\\
    & \overline{\rho}_j
\end{pmatrix}$ are non-split. 
\end{prop}

A condition that is preserved when taking subquotients such as Fontaine--Laffaille at $p$ and respective Selmer group conditions allow us to remove this ambiguity, and show that for any pair $(i,j)$, there is a lattice yielding a non-split subquotient 
$\begin{pmatrix}
    \overline{\rho}_i & *\\
    & \overline{\rho}_j
\end{pmatrix}$, 
i.e., that any vertices of the graph $\Gamma$ are connected by an oriented edge in both directions.

\section{$R = \mathbf{T}$ theorem}{\label{R=T section}}
In this section, we formulate sufficient conditions to prove an $R = \mathbf{T}$ theorem, provided enough congruences among automorphic forms, even if $R$ is not a discrete valuation ring.

Let $S$ be the set of representations $\rho_{f}: G_{\Sigma} \to \GL_n(E)$ such that $\overline{\rho}_{f}^{\sss} \cong \overline{\sigma}^{\sss}$ and that $\rho_{f}$ is irreducible and Fontaine--Laffaille at $p$. Then by Proposition \ref{Lattice}, there exists a $\G$-stable lattice such that
$$
\overline{\rho}_{f} \cong
\begin{pmatrix}
    \rho_1 & *_1 & *_2 \\ 
    0 & \rho_2 & *_3 \\ 
    0 & 0 & \rho_3 \\ 
\end{pmatrix},
$$
and 
$
\begin{pmatrix}
    \rho_1 & *_1\\
    0 & \rho_2
\end{pmatrix},
\begin{pmatrix}
    \rho_2 & *_3\\
    0 & \rho_3
\end{pmatrix}$ are non-split extensions. Such residual representation is unique under Assumption \ref{main assumptions} (1)(2) (cf. Proposition \ref{Uniqueness of residual}), so $\overline{\rho}_{f} \cong \overline{\sigma}$, and we obtain an $\mO$-algebra map by representability of $R$,
$$
\phi : R \to \prod_{\rho_{f} \in S} \mO.
$$
We write $\mathbf{T} = \phi(R)$.

\begin{thm}{\label{R=T}}
Suppose $S$ is finite. Assume Assumptions \ref{main assumptions} and $\# H^1_{\Sigma}(F,\Hom(\Tilde{\rho_2},\Tilde{\rho_1}) \otimes_{E} E/\mO) \le \# \mO/\varpi^r \mO$, $\# \mathbf{T}/ \phi(I^{\tot}) \ge \# \mO/\varpi^r$. Then $\phi: R \to \mathbf{T}$ is an isomorphism.
\end{thm}
\begin{proof}
Following the arguments in Theorem \ref{main thoerem}, we get $\# R/I^{\tot} \le \# \mO/\varpi^r$. Thus we get that $\phi$ induces an isomorphism $\overline{\phi}: R/I^{\tot} \to \mathbf{T}/\phi(I^{\tot})$. And Theorem 4.1(i) in \cite{berger2013deformation} implies that $\phi$ is an isomorphism.
\end{proof}

\begin{rmk}
In applications, $\mathbf{T}$ will be identified with some local complete Hecke algebra, and $\rho_f$ will be the Galois representation attached to an automorphic representation $f$ and $S$ will be the subset of ($L$-packets of) automorphic representations $f$ whose associated Galois representation $\rho_f$ satisfies our deformation conditions. The condition on $\# \mathbf{T}/\phi(I^{\tot})$ will be a consequence of a lower bound on $\# \mathbf{T}/J$, where $J$ is some relevant congruence ideal. Such a lower bound of $\# \mathbf{T}/J$ can be achieved by congruences of automorphic forms. See Section \ref{Ikeda lifts} for examples and details.
\end{rmk}

\begin{rmk}
Theorem \ref{R=T} utilizes the numerical criterion due to Bergen--Kloin \cite[Theorem 4.1]{berger2013deformation}, which is an alternative of the well-known Wiles--Lenstra numerical criterion \cite{de1997criteria}. Unlike the Wiles--Lenstra criterion, which identifies an order related to the augmentation ideal of $R$, the Bergen--Kloin identifies an order related to a principal ideal of $R$. In many cases, the principal ideal is identified to be the reducbility ideal. 
\end{rmk}

\begin{rmk}{\label{Selma not cyclic}}
In Theorem \ref{R=T}, if $\# H^1_{\Sigma}(F,\Hom(\Tilde{\rho_2},\Tilde{\rho_1}) \otimes_{E} E/\mO) = \# \mathbf{T}/J = \# \mO/\varpi$, by Lemma \ref{torsion Selmer is Selmer torsion}, then it implies $\dim_{k} H^1_{\Sigma}(F,\Hom(\rho_2,\rho_1) = 1$ in Assumption \ref{main assumptions}. In this case, $R$ and $\mathbf{T}$ will be identified as a DVR. 

However, it is certainly possible that $\# H^1_{\Sigma}(F,\Hom(\Tilde{\rho_2},\Tilde{\rho_1}) \otimes_{E} E/\mO) \ge \# \mO/\varpi$ and $\dim_{k} H^1_{\Sigma}(F,\Hom(\rho_2,\rho_1) = 1$. For example, in the case of elliptic curves, let $E/\Q$ be the elliptic curve with conductor 66 and minimal Weierstrass equation,
$$
y^2=x^3-456219x-118606410.
$$
Then using MAGMA, we get $\text{Sel}_{2}(E) \cong (\Z/2)^3$ and $\dim \text{Sel}_2(E)[2] = 1$.
\end{rmk}

\section{Abelian Surfaces with Rational $p$-isogeny}\label{Abelian surface}

Let $A_{/ \Q}$ be an abelian surface over $\Q$. It is called a \textit{QM abelian surface} if $\End_{\Q} A$ is an order in the non-split quaternion algebra over $\Q$. In analogy with the Taniyama-Shimura Conjecture, in \cite{brumer2014paramodular}, Brumer and Kramer proposed a precise conjecture of a one-to-one correspondence between isogeny classes of abelian surfaces $A_{/ \Q}$ of conductor $N$ with $\text{End}_{\Q} A = \Z$ or QM abelian surfaces $A_{/ \Q}$ of conductor $N^2$, and cuspidal, nonlift weight 2 Siegel modular forms $f$ on the level $K(N)$ with rational eigenvalues, up to scalar multiplication. Here $K(N)$ is the paramodular group of level $N$ defined by $K(N) = \gamma M_4 (\Z) \gamma^{-1} \cap Sp_4(\Q)$ with $\gamma = \text{diag}[1, 1, 1, N]$.
\footnote{The authors later made revisions to the conjecture. See \cite{https://doi.org/10.48550/arxiv.1004.4699} for details.}
The conjecture is known as the Paramodular Conjecture.

There has been important progress in verifying this conjecture. The first verified surface is of prime conductor $277$ by \cite{brumer2019paramodularity} and there are no smaller conductors on either side of the correspondence for all prime $N < 277$ \cite[Proposition 1.5]{brumer2014paramodular} \cite[Theorem 1.2]{poor2015paramodular}. In \cite{boxer2021abelian}, the authors showed $A$ is potentially modular under the assumptions that the residual Galois representation $\overline{\rho}_A$ from $A(\overline{\Q})[p]$ has image containing $\text{Sp}_4(\F)$ and that $\overline{\rho}_A$ is modular, proving the conjecture for infinitely many cases. In particular, \cite{calegari2020some} applied the main theorem of \cite{boxer2021abelian} and provided explicit examples of modular abelian surfaces, e.g., curves with conductor $2^4 5^3 7^2, 2^{15} 5$.
In \cite{berger2020deformations}, the authors studied the modularity of surfaces with rational torsions and square-free conductors and proved the modularity of the abelian surface of square-free level $731$ as an example. 

In this section, we consider abelian surfaces with a rational $p$-isogeny and $\End_{\Q} A = \Z$, a more general case than the one studied in \cite{berger2020deformations}, who assumed that $A(\Q)[p] \ne 0$. We also assume that $A$ has polarization degree prime to $p$. We first show that the residual representation arising from such a surface $A$ has the configuration as in Theorem \ref{main thoerem}, and establish sufficient conditions for the universal deformation ring to be a DVR. We also discuss the implications of the $\lambda$-part of the Bloch--Kato conjecture to our result, relating special $L$-values to the structure of the universal deformation ring.

\subsection{Matrix Form of Residual Galois Representation of Abelian Surface with Rational Isogeny}

\begin{defn}
An abelian surface $A_{/ \Q}$ is said to have a \textbf{rational $n$-isogeny} (or a cyclic isogeny of degree $n$) if $A$ has a cyclic subgroup of order $n$ defined over $\Q$ that is stable under the action of $\Gal(\overline{\Q}/\Q)$.
\end{defn}
In particular, any surface with a rational torsion of order $n$ has a rational $n$-isogeny. 

In this section, 
let $A$ be an abelian surface with conductor $N$. Suppose $A$ has a rational $p$-isogeny and a polarization of degree prime to $p$ such that $p \nmid N$ and $\End_{\Q} A = \Z$. Let $\Sigma = \{p\} \cup \{\ell | \ell \mid N\}$. Let $T_p(A)$ be the $p$-adic Tate module of $A$. Then $V_p(A) := T_p(A) \otimes E$ gives rise to a Galois representation $\sigma_A: \G \to \GL_4(E)$. We will consider the case when $\End_{E[G_{\Q}]}(V_p(A)) = \End_{\Q}(A) \otimes E = E$. Then $\sigma_A$ is absolutely irreducible as it is semisimple (cf. \cite{faltings1983endlichkeitssatze} or \cite{faltings1986finiteness}).

As $A$ has a rational $p$-isogeny, we know that the semi-simplication of residual Galois representation $\overline{\sigma}_A: \G \to \text{GL}_{4}(\F)$ attached to $A$, $\overline{\sigma}_A^{ss}$ has a Jordan-H\"{o}lder factor, $\psi: \G \to \F^{\times}$. Let $\chi$ be the mod-$p$ cyclotomic character. Using Weil pairing, one obtains the following Proposition describing the semi-simplication of $\overline{\sigma}_A$.

\begin{prop}{\label{A ss}}
One has $\overline{\sigma}_A: \G \to \text{GL}_{4}(\F)$ satisfies
\begin{equation}
   \overline{\sigma}_A^{\sss} \cong \psi \oplus \rho^{\sss} \oplus \psi^{-1} \chi
\end{equation}
for some character $\psi:  \G \to \F^{\times}$ and $\rho: G_{\Sigma} \to \GL_2(\F)$.
\end{prop}



The character $\psi$ is unramified outside of $p N$, thus $\psi = \phi \chi^{i}$ for some character $\phi$ unramified at $p$ and $i \in \{0,1\}$ as these are the only powers of $\chi$ realized by finite flat group schemes over $\Z_p$. Suppose $i = 1$, then $\psi^{-1} \chi = \phi^{-1}$ is unramified at $p$. Thus without loss of generality, we can assume $\psi$ is unramified at $p$ and odd. Also, it follows from Proposition \ref{A ss} that if $\psi$ is non-trivial, then the conductor $N$ is not square-free as $\text{cond}(\psi)^2 \mid N$. 

Recall that $\epsilon$ is the $p$-adic cyclotomic character. Assume that $\rho$ is absolutely irreducible. As $\sigma_A$ has determinant $\epsilon^2$, we have $\det \rho = \chi$, and
it follows that $\rho$ is odd. Note by Serre’s modularity conjecture (proved in \cite{khare2009serre}), one has $\rho = \overline{\rho}_f$ where $\rho_f: \G \to  \GL_{2}(\Q_p)$ is the Galois representation attached to a newform $f \in S_2(N_f)$ for some level $N_f$, and $\overline{\rho}_f$ is its mod $p$ reduction.

As $\sigma_A$ is irreducible, if $H^{1}_{\Sigma}(\Q,\psi^{2}\chi^{-1}) = 0$, by Proposition \ref{Lattice}, we can choose a $\G$-stable lattice such that residual Galois representation $\overline{\sigma}_{A}: G_{\Sigma} \to \GL_{4}(\F)$ satisfies
\begin{equation}{\label{residual matrix form}}
    \overline{\sigma}_{A} \cong 
    \begin{pmatrix}
        \psi & a & b \\
        0  & \rho & c \\
        0  & 0 & \psi^{-1} \chi \\
    \end{pmatrix},
\end{equation}
and $\begin{pmatrix}
    \psi & a\\
    0 & \rho
\end{pmatrix},
\begin{pmatrix}
    \rho & c\\
    0 & \psi^{-1} \chi
\end{pmatrix}$ are non-split extensions. Furthermore, define $\tau: \Z_p[\G] \to \Z_p[\G]$ by $\tau(g) = \epsilon(g) g^{-1}$. Then $\tau$ is an anti-involution interchanging $\psi$ with $\psi^{-1} \chi$ and leaving $\rho$ fixed. 

By the discussion above, $\overline{\sigma}_A$ is as in the setup of Section \ref{general residual set up}. 

\begin{rmk}
As
$$
\Hom(\rho,\psi) = \rho^{\vee} \otimes \psi = \rho \otimes (\psi \chi^{-1}) = (\psi^{-1} \chi)^{\vee} \otimes \rho = \Hom(\psi^{-1} \chi, \rho),
$$
one has $a \rho^{-1}, c \psi \chi^{-1} \in H^{1}(\Q, \rho(\psi^{-1} \chi))$. 
\end{rmk}

\subsection{Examples}
Thanks to Shiva Chidambaram, who helped us compute mod-$5$ Galois representations of principally polarized abelian surfaces with a $1$-dimensional Galois-stable subspace, we were able to compute examples of principally polarized abelian surfaces with rational $5$-isogenies and specific descriptions of $\psi$ when $p = 5$. All surfaces obtained are Jacobians of genus 2 curves. Finally, we know by the algorithm that only characters $\psi, \psi^{-1}\chi$ occur in $\overline{\sigma}_A^{\sss}$, thus $\rho$ is irreducible in these examples.
{\footnote{Please find the MAGMA code at \textit{https://github.com/cocoxhuang/Find-One-Dimension-Character/tree/main}}}
We also know  $\End_{\Q}(A) = \Z$ by consulting LMFDB \cite{lmfdb}. 
Please see Table 1 below.

\begin{table}[h!]{\label{Table 1}}
    \centering
    $$\bm{p = 5}$$
    \begin{tabular}{ |c|c|c|c| }
        \hline
        \hline
        $N$ & LMFDB Label & $\Q$-Torsion & Conductor of $\psi$\\
        \hline
        8960 = $2^{8} \cdot 5 \cdot 7  $ & 8960.c.17920.1 & $\Z/2\Z\oplus\Z/2\Z$ & 16\\
        13351 = $13^{2} \cdot 79  $ & 13351.a.173563.1 & $\Z/2\Z$ & 13 \\
        15379 = $7 \cdot 13^{3}  $ & 15379.a.107653.1 & trivial & 13\\ 
        26112 = $2^{9} \cdot 3 \cdot 17  $ &  26112.a.26112.1 & trivial & 16 \\ 
        45568 = $2^{9} \cdot 89  $ & 45568.a.45568.1 & trivial & 16\\ 
        54043 = $11 \cdot 17^{3}  $ & 54043.a.54043.1 & trivial & 17\\ 
        108086 = $2 \cdot 11 \cdot 17^{3} $ & 108086.a.864688.1 & $\Z/2\Z$ & 17\\
        146944 = $2^{9} \cdot 7 \cdot 41  $ & 146944.b.293888.1 & $\Z/3\Z$ & 16\\
        191607 = $3 \cdot 13 \cdot 17^{3}  $ & 191607.b.191607.1 & trivial & 17\\
        230911 = $17^{3} \cdot 47 $ & 230911.a.230911.1 & trivial & 17\\
        321280 = $2^{8} \cdot 5 \cdot 251 $ & 321280.e.652560.1 & trivial & 16\\
        349696 = $2^{9} \cdot 683 $ & 349696.a.349696.1 & trivial & 16\\
        355914 = $2 \cdot 3^{4} \cdot 13^{3}  $ & 355914.a.355914.1 & trivial & 39\\ 
        454912 = $2^{8} \cdot 1777 $ & 454912.p.909824.1 & trivial & 16\\
        569023 = $7 \cdot 13^{3} \cdot 37 $ & 569023.a.569023.1 & trivial & 13\\
        889253 = $17^{3} \cdot 181 $ & 889253.a.889253.1 & trivial & 17\\
        
        \hline
        \end{tabular}
    \caption{Examples of principally polarized abelian surfaces with rational cycles}
\end{table}

\subsection{$R$ is a DVR}{\label{section $R$ is a DVR}}

We are now ready to state our main application of the results in Section \ref{R is DVR theorem section} to abelian surfaces. The following Theorem is an easy consequence of Theorem \ref{main thoerem}.

\begin{thm}{\label{R is DVR}}
Let $A$ be an abelian surface with conductor $N$ and a rational $p$-isogeny where $p \nmid N$ and $q \not \equiv 1 \mod p$ for all $q \in \Sigma$ and $\End_{\Q} A = \Z$. $\overline{\sigma}_A^{\sss} \cong \psi \oplus \overline{\rho}_f \oplus \psi^{-1} \chi$ for some character $\psi: \G \to \F^{\times}$. Suppose $\overline{\rho}_f$ is absolutely irreducible. Let $\Tilde{\psi}$ be the Teichm\"uller lift of $\psi$. Assume there is a unique Fontaine--Laffaille deformation of $\overline{\rho}_f$ to $\GL_2(\Z_p)$ up to strict equivalence class, and 
\begin{enumerate}
    \item $H^{1}_{\Sigma}(\Q,\psi^{2}\chi^{-1}) = 0$;
    \item $\dim_{\F} H^{1}_{\Sigma}(\Q,\overline{\rho}_f(\psi^{-1}))= 1$;
    \item $\# H^1_{\Sigma}(\Q,\rho_f(\tpsi \epsilon^{-1}) \otimes E/\mO) \le p$.
\end{enumerate}
Then $R \cong \Z_p$ and there is a unique lift $\sigma_A$ of $\sigma$ that is Fontaine--Laffaille at $p$.
\end{thm}
\begin{proof}
It remains to check if conditions (1)--(4) in Theorem \ref{main thoerem} are satisfied. $\overline{\sigma}$ is Fontaine-Laffaille at $p$ as $p \nmid N$ (cf. \cite{serre1968good} and \cite{fontaine1982certains}), so condition (1) is satisfied. The assumptions on $\overline{\rho}_f$ and the primes in $\Sigma$ guarantee condition (2) is satisfied (cf. \cite[Proposition 9.5]{berger2013deformation}). The existence of lattice provided by Proposition \ref{Lattice} along with assumption (1) shows condition (3) is satisfied. The anti-involution $\tau: \Z_p[\G] \to \Z_p[\G]$ by $\tau(g) = \epsilon(g) g^{-1}$ guarantees condition (4).
\end{proof}

\begin{cor}(Unique isogeny classes)
Let $\psi$ be as in Theorem \ref{R is DVR} and assume its conditions hold, then there is a unique isogeny class of abelian surfaces $A/_{\Q}$ of conductor $N$ with a rational $p$-isogeny and $\overline{\sigma}_A^{\sss} \cong \psi \oplus \overline{\rho}_f \oplus \psi^{-1} \chi$. 
\end{cor}

\begin{rmk}{\label{R_i = O}}
To verify there is a unique Fontaine--Laffaille at $p$ deformation of $\rho_f$ to $\GL_2(\Z_p)$ up to strict equivalence class, assuming the Fontaine-Mazur conjecture, it suffices to check that there is no newform $h$ such that its associated Galois representation $\overline{\rho}_{h} \cong \overline{\rho}_{f}$ following a discussion in 10.2 in \cite{berger2013deformation}. If such $h$ exists, it again should be a newform in $S_{2}(N_f)$ following Theorem 0.2 in \cite{livne1989conductors}. Thus a sufficient condition for a unique lift of $\overline{\rho}_f$ is that there is no newform $h \in S_{2}(N_f)$ such that $f \equiv h (\mmod p)$.
\end{rmk}

In what follows we discuss situations when the conditions of Theorem \ref{R is DVR} are satisfied.

\subsubsection{\textsf{Conductor $N = 13351, 15379, 26112, 45568, 191607$}}
We will show that some abelian surfaces with such conductors satisfy all assumptions in Theorem \ref{R is DVR} except for conditions (1)–(3), as verifying these conditions is currently beyond the author's reach. We will examine the surface with conductor $15379$ in detail. The arguments for the other surfaces are similar.

By Table 1, there is an abelian surface $A$ of conductor $15379 = 7 \cdot 13^3$, which arises as the Jacobian of the hyperelliptic curve with the equation (see \cite{lmfdb})
$$
y^2 + x^3 y = x^5 - 4x^3 - x^2 + 5x - 2,
$$
which has a rational $5$-isogeny and is principally polarised. And $\overline{\sigma}_{A}^{\sss}: \G \to \GL_4(\F) = \psi \oplus \rho \oplus \psi^{-1} \chi$, where the conductor of $\psi$ is $13$ and $\rho$ is irreducible. By Serre's conjecture, $\rho = \overline{\rho}_f$ for some cuspform $f$ of weight 2 and level 91. By comparing $\tr(\overline{\sigma}_A)(\Frob_{\ell})$ and $\psi(\ell) + a_{\ell}(f) + \psi^{-1}\chi(\ell)$ for various primes $\ell$ using MAGMA, we know that $f$ is the newform of LMFDB label 91.2.a.a with coefficient field $\Q$. It corresponds to an elliptic curve $E'$ with LMFDB label 91.a1 and equation 
$$
y^2+y=x^3+x.
$$
As one can check from LMFDB that $5$ is not a congruence prime for $f$, we have $R_{\rho} = \Z_p$ (Remark \ref{R_i = O}). Please see Table 2 for the subset of surfaces $A$ in Table 1 such that $R_{\rho} = \Z_p$ and their corresponding modular forms $f$.

\begin{table}[h!]{\label{Table 2}}
    \centering
    \begin{tabular}{ |c|c|c|c| }
        \hline
        \hline
        $N$ & LMFDB Label of $A$ & LMFDB Label of $f$ & Coefficient Field of $f$\\
        \hline
        13351 & 13351.a.173563.1 & 79.2.a.a & $\Q$\\
        15379 & 15379.a.107653.1 & 91.2.a.a & $\Q$\\ 
        26112 &  26112.a.26112.1 & 102.2.a.c & $\Q$\\ 
        45568 & 45568.a.45568.1 & 178.2.a.b & $\Q$\\         
        191607 & 191607.b.191607.1 & 663.2.a.a & $\Q$\\
        \hline
        \end{tabular}
    \caption{Examples of abelian surfaces with $R_{\rho} = \Z_p$}
\end{table}

\begin{rmk}
Though all surfaces in Table 2 have coefficient field $\Q$, it is only needed that the coefficient field $K$ of the surface has a prime $v \mid p$ such that $K_{v} \cong \Q_p$.
\end{rmk}

\subsubsection{\textsf{Selmer group conditions}}
We further identify the conditions under which the Selmer group conditions in Theorem \ref{R is DVR} hold. First, we have the following Proposition guaranteeing condition (1) using certain Bernoulli numbers.

\begin{prop}
Let $\eta(\tpsi^{-2},\Sigma) = B_2(\tpsi^{-2}) \cdot \prod_{\ell \in \Sigma \setminus \{p\}} (1 - \tpsi^{-2}(\ell) \ell^{2})$. Then 
$$
\# H^1_{\Sigma}(\Q, E \otimes \mO (\Tilde{\psi}^{2}\epsilon^{-1})) \le \# \mO/\eta(\tpsi^{-2},\Sigma).
$$
Particularly, if $\text{val}_p(\eta(\tpsi^{-2},\Sigma)) = 0$, then $H^1_{\Sigma}(\Q, \psi^{2}\chi^{-1}) = 0$.
\end{prop}
\begin{proof}
It follows from Proposition 5.7 in \cite{berger2019modularity}.
\end{proof}


As far as conditions (2)(3) are concerned, we can relate $H^{1}_{\Sigma}(\Q, \cdot)$ to $H^{1}_{f}(\Q, \cdot)$, which conjecturally are bounded by $L$-values. See Remark \ref{Selmer and L} for more details. Let $w_{f,\ell} \in \{-1,1\}$ be the eigenvalue of the Atkin-Lehner involution at $\ell$ corresponding to $f$. We have the following Lemma identifying $H^{1}_{\Sigma}(\Q, \cdot)$ with $H^{1}_{f}(\Q, \cdot)$ when the level $N_f$ is square-free.

\begin{lem}{\label{squarefree surface selmer groups}}
Suppose $N_f$ is a square-free integer. Suppose $\ell \not \equiv 1 (\mmod p)$ for every $\ell \mid N_f$. Suppose $\tpsi$ is either tamely ramified at $\ell$, or unramified at $\ell$ and $p \nmid 1 + w_{f,\ell} \ell^{2} \Tilde{\psi}^{-1}(\ell)$, for all $\ell \mid N_f$, then 
\begin{equation}{\label{first Sigma = f}}
    H^1_{\Sigma}(\Q,\rho_f(\Tilde{\psi}\epsilon^{-1})) \otimes E/\mO) = H^1_{f}(\Q,\rho_f(\Tilde{\psi}\epsilon^{-1})) \otimes E/\mO)
\end{equation}
\end{lem}
\begin{proof}
Let $V = \rho_f(\Tilde{\psi}\epsilon^{-1})$. By Lemma \ref{two selmer groups}, it suffices to check that $P_{\ell}(V^*,1) \in \mO^{*}$ and $W^{I_{v}}$ is divisible where $W = V/T$ and $T \subset V$ is a $\G$-stable $\mO$ lattice.

As $N_f$ is square-free, we know that for $\ell \mid N_f$,
$$
\rho_f|_{G_{\Q_{\ell}}} \cong_{E} 
\begin{pmatrix}
    \gamma \epsilon & *\\
     & \gamma
\end{pmatrix}  
$$
and $\rho_f|_{I_{\ell}} \ne 1$, where $\gamma: G_{\Q_{\ell}} \to E^{\times}$ is an unramified character sending $\Frob_{\ell}$ to $a_{\ell}(f)$ (cf. \cite[~Theorem 3.26(3)(b)]{hida2000modular}). 

$V^* = V^{\vee}(1) = \rho_f(\Tilde{\psi}^{-1} \epsilon)$. Suppose $\tpsi$ is unramified at $\ell$, then $(V^{*})^{I_{\ell}} = \gamma \Tilde{\psi}^{-1} \epsilon^2$ and 
$$
P_{\ell} (V^*,1) = \det (1 - \Frob_{\ell}|_{(V^*)^{I_{\ell}}}) = 1 - a_{\ell}(f) \ell^{2} \Tilde{\psi}^{-1}(\ell).
$$
Theorem 3 in \cite{Atkin1970HeckeOO} shows that $a_{\ell} (f) = -w_{f,l}$ with Hecke action normalized as in  \cite[~Theorem 4.6.17(2)]{miyake2006modular}. Then by assumption $p \nmid 1 + w_{f,\ell} \ell^{2} \Tilde{\psi}^{-1}(\ell)$, $P_{\ell}(V^*,1) \in \mO^{*}$. 

Suppose $\tpsi$ is tamely ramfied at $\ell$, then $(V^{*})^{I_{\ell}}$ is trivial, and $P_{\ell}(V^*,1) = 1 \in \mO^{*}$ as well. Since $\ell \not \equiv 1 (\mmod p)$, $\rho_f$ is tamely ramified at $\ell$. Thus, in either case of $\tpsi$'s ramification, following a similar proof as in \cite[~Proposition 2.5]{berger2020deformations}, we get that $W^{I_{\ell}}$ is divisible and \eqref{first Sigma = f} holds. 
\end{proof}

However, generally, $N_f$ is not guaranteed to be square-free. In this case we prove the following Proposition relating $H^1_{\Sigma}(\Q, \cdot)$ and $H^1_{f}(\Q, \cdot)$ using the local components of the newform $f$.

\begin{prop}{\label{general Selmer group relation}}
Suppose for any $\ell \mid N_f$, we have $\ell \not \equiv 1 (\mmod p)$ and one of the following holds:
\begin{enumerate}
    \item $\val_{\ell}(N_f) = 1$. $\tpsi$ is either tamely ramified at $\ell$, or unramified at $\ell$ and $p \nmid 1 + w_{f,\ell} \ell^{2} \Tilde{\psi}^{-1}(\ell)$.
    \item $\val_{\ell}(N_f) > 1$, $a_{\ell}(f) \ne 0$ and $\tpsi$ is tamely ramified at $\ell$.
\end{enumerate}
Then 
\begin{equation}{\label{first Sigma = f squarefull}}
    H^1_{\Sigma}(\Q,\rho_f(\Tilde{\psi}\epsilon^{-1})) \otimes E/\mO) = H^1_{f}(\Q,\rho_f(\Tilde{\psi}\epsilon^{-1})) \otimes E/\mO).
\end{equation}

\end{prop}
\begin{proof}
As in Lemma \ref{squarefree surface selmer groups}, we will verify that $P_{\ell}(V^*,1) \in \mO^{*}$ and $W^{I_{v}}$ is divisible for all $\ell \mid N_f$. The case where $\val_{\ell}(N_f) = 1$ was discussed in Lemma \ref{squarefree surface selmer groups}. If $a_{\ell}(f) \ne 0$ and $\val_{\ell}(N_f) > 1$, then $\rho_f|_{G_{\Q_{\ell}}}$ is isomorphic
to the principal series $\eta_1 \oplus \eta_2$ (cf. Proposition 2.8 and Section 5 in \cite{loeffler2012computation}), where $\eta_1$ is unramified and $\eta_2$ is tamely ramified as $\ell \not \equiv 1 (\mmod p)$. Let $V = \rho_f(\Tilde{\psi}\epsilon^{-1})$. 
As $\tpsi$ is ramified, $(V^{*})^{I_{\ell}}$ is trivial, and $W^{I_{\ell}}$ is trivial and divisible by local class field theory. 
\end{proof}

\begin{rmk}{\label{Selmer and L}}
The $\lambda$-part of Bloch--Kato conjecture predicts that for a $\G$-module $V$ valued in $E$ with ring of integer $\mO$ and uniformizer $\varpi$, 
$$
\# H^1_{f}(\Q,V \otimes_{\mO} E/\mO) = (\# \mO/\varpi \mO)^m,
$$
where $m = \text{val}_{\varpi} (L^{\text{alg}} (V^{\vee}(1),0))$ and $L^{\text{alg}}(\cdot, 0)$ are certain normalized algebraic special $L$-values (see Corollary 7.8 in \cite{klosin2009congruences} for details). One could check the bounds on the Selmer groups through numerical bounds on certain $L$-values using Proposition \ref{general Selmer group relation}. For more details of the $\lambda$-part of Bloch--Kato conjecture, see Conjecture 2.14 in \cite{diamond2004tamagawa} and Conjecture 9.14 in \cite{klosin2009congruences}.
\end{rmk} 

\section{Ikeda Lifts}\label{Ikeda lifts}
To establish Theorem \ref{R=T}, one usually needs to provide a lower bound on  $\# \mathbf{T}/J$ for some congruence ideal $J \subset \mathbf{T}$, and such a lower bound can be achieved by the existence of congruences between automorphic forms.  In this section, provided such a congruence, we apply Theorem \ref{R=T} to establish an $R = \mathbf{T}$ theorem. Specifically, in Theorem 6.1 in \cite{brown2020congruence}, the authors provided sufficient conditions for an eigenform $f'$ defined on the unitary group U$(n,n)(\adele_{\Q})$ of full-level to be orthogonal and congruent to an Ikeda lift $I_{\phi}$ on the same group, and constructed such $f'$. When $n = 3$, the residual Galois representation attached to $f'$ is conjectured to have three Jordan-H\"{o}lder factors and is of the form \eqref{reducible} by Proposition \ref{Lattice}. Furthermore, we identify sufficient special L-value conditions for the $R = \mathbf{T}$ theorem in place of the Selmer group conditions in Theorem \ref{R=T}, assuming $\lambda$-part of the Bloch-Kato conjecture.

\subsection{Hermitian Automorphic forms}
Let $K = \Q(\sqrt{-D_K})$ be an imaginary quadratic extension of $\Q$ with discriminant $D_K$ such that the conditions related to $K$ in Theorem 6.1 in \cite{brown2020congruence} are satisfied (see Lemma \ref{Existance of congruence} for further discussion on the technical conditions). Let $\chi_K$ denote the quadratic character associated with the imaginary quadratic extension $K/\Q$. Let $\mathbf{G}_m$ be the multiplicative group scheme. Associated to $K$, there is the unitary similitude group scheme over $\Z$:
$$
\text{GU}(n,n) = \{A \in \text{Res}_{\mO_{K}/\Z} \GL_{2n/\mO_K} : A J_n A^* = \mu_n(A) J_n \}
$$
where $J_n = \begin{pmatrix}
    & - I_n\\
    I_n &
\end{pmatrix}$ and $1_n$ is the $n \times n$ identity matrix. Here $\mu_n: \text{Res}_{\mO_{K}/\Z} \GL_{2n/\mO_K} \to \mathbf{G}_{m/\Z}$ is a morphism of $\Z$-group schemes and is the similitude factor. We define $U(n,n)$ by the following exact sequence:
$$
1 \rightarrow \text{U}(n,n) \rightarrow \text{GU}(n,n) \xrightarrow{\mu_n} \GL_1 \rightarrow 1.
$$

Let $\mathcal{M}_{n,k''}(\hat{\Z})$ be the space of automorphic forms on U$(n,n)$ of weight $k''$ and full level and let $\mathcal{S}_{n,k''}(\hat{\Z})$ denote the space of cusp forms in $\mathcal{M}_{n,k''}(\hat{\Z})$. For a more detailed treatment, see \cite{brown2020congruence} Section 2 or \cite{Eischen22}.

For a $2n \times 2n$ matrix $g$, we will write $a_{g}, b_{g}, c_{g}, d_{g}$ to be the $n \times n$ matrices defined by 
$$
g = 
\begin{pmatrix}
    a_{g} & b_{g}\\
    c_{g} & d_{g}
\end{pmatrix}.
$$

Let $\mathfrak{N} \subset \Z$ be an ideal. For $v \in \f$ we set 
\begin{equation}
    \mathcal{K}_{0,n,v}(\mathfrak{N}) = \{g \in G_n(F_v): a_g,b_g, d_g \in \text{Mat}_n(\mO_{K,v}), c_g \in \text{Mat}_n(\mathfrak{N} \mO_{K,v}) \}
\end{equation}

Denote the set of finite places of $\Q$ as $\f$. Let $\adele_{\Q}$ denote the adeles of $\Q$ and $\adele_{\Q,\f}$ is the finite part of $\adele_{\Q}$. 
Let $F_{\textbf{a}} = \mathbf{R} \otimes_{\Q} F$ and $\mathbf{R}^{\textbf{a}} = \prod_{v \in \textbf{a}} \mathbf{R}$ via the map $a \to (\sigma(a))_{\sigma \in \textbf{a}}$.

Set
$$
\mathcal{K}^{+}_{0,n,\infty} = 
\bigg\{ 
\begin{pmatrix}
    A & B \\
    -B & A
\end{pmatrix} 
\in G_n(\mathbf{R}): A,B \in \GL_n(\mathbf{C}), A A^{*} + B B^{*} = 1_n, AB^{*} = BA^{*}
\bigg\}
$$
and $\mathcal{K}_{0,n,\infty}$ to be the subsgroup of $G_n(\mathbf{R})$ generated by $\mathcal{K}^{+}_{0,n,\infty}$ and $J_n$, where $J_n = \begin{pmatrix}
    & - 1_n\\
    1_n &
\end{pmatrix}$.

Define 
\begin{equation}
    \textbf{H}_n = \{Z \in \text{Mat}_n(\mathbf{C}): -i 1_n (Z - Z^{*}) > 0\}
\end{equation}
where $Z^{*} = \overline{Z}^t$ and bar denotes the action of the nontrivial element of $\Gal(K/\Q)$.

Let $g_{\infty} \in G_n(\mathbf{R})$. As $g_{\infty}$ is an $2n \times 2n$ matrix, we will write $a_{g_{\infty}}, b_{g_{\infty}}, c_{g_{\infty}}, d_{g_{\infty}}$ to be the $n \times n$ matrices defined by 
$$
g_{\infty} = 
\begin{pmatrix}
    a_{g_{\infty}} & b_{g_{\infty}}\\
    c_{g_{\infty}} & d_{g_{\infty}}
\end{pmatrix}
$$
and $Z \in \upperhalf_n$, set $j(g_{\infty},Z) = \det(c_{g_{\infty}} Z + d_{g_{\infty}})$.

Let $\mathcal{K}$ be an open compact subsgroup of $G_n(\mathbf{A}_{\Q,\f})$. For $k \in \Z$, let $\mathcal{M}_{n,k}(\mathcal{K})$ denote the $\mathbf{C}$-space of functions $f: G_n(\adele_{\Q}) \to \mathbf{C}$ satisfying the following 
\begin{enumerate}
    \item $f(\gamma g) = f(g)$ for all $\gamma \in G_n(\Q), g \in G_n(\adele_{\Q})$,
    \item $f(g \kappa) = f(g)$ for all $\kappa \in \mathcal{K}, g \in G_n(\adele_{\Q})$,
    \item $f(g u) = j(u, i1_n)^{-k} f(g)$ for all $g \in G_n(\adele_{\Q}), u \in \mathcal{K}_{0,n,\infty}$,
    \item $f_c(Z) = j(g_{\infty}, i1_n)^k f(g_{\infty} c)$ is a holomorphic function of $Z = g_{\infty} i 1_n \in \mathbf{H}_n$ for every $c \in G_n(\adele_{\Q,\f})$. 
\end{enumerate}
Let $\mathcal{S}_{n,k}(\mathcal{K})$ denote the space of cusp forms in $\mathcal{M}_{n,k}(\mathcal{K})$.

\subsection{Ikeda lifts}



From now on let $n = 3$. Let $p \nmid 2 D_K, p > 2k'+2, p \nmid \# \text{Cl}_K$ be a prime. Note by assumption, $p$ is unramified in $K$.

Let $\phi$ be a newform in $S_{2k'}(1)$, where $S_{2k'}(1)$ is the space of (classical) elliptic cusp forms of weight $2k'$, level $1$, such that no newform $S_{2k'}(1)$ that $\phi \equiv h (\mmod \varpi)$ and $\phi$ is ordinary at $p$. Moreover, assuming the $L$-value conditions related to $\phi$ such that $\val_{\varpi}(\eta_{\phi}) \ge 0$ and $\val_{\varpi}(\mathcal{U}) = 0$ in Theorem 6.1 in \cite{brown2020congruence}.
Let $\rho_{\phi}: G_{\Q} \to \GL_2(E)$ be the Galois representation associated to $\phi$. By choosing some $G_{\Sigma}$-invariant lattice, we can assume $\rho_{\phi}: G_{\Q} \to \GL_2(\mO)$. We further assume the residual Galois representation $\overline{\rho}_{\phi}|_{G_K}: G_{K} \to \GL_2(\overline{\F})$ is absolutely irreducible. 

Ikeda has shown that there exists $I_{\phi} \in \mathcal{S}_{3,2k'+2}(\hat{\Z})$ (the Ikeda lift of $\phi$), a certain eigenform whose $L$-function is closely related to the $L$-function of $\phi$. It conjecturally satisfies that its associated Galois representation  $\rho_{I_{\phi}}: G_K \to \GL_{6}(\overline{\Q}_p)$ is
$$
\rho_{I_{\phi}} \cong \epsilon^{2-k'} \bigoplus_{i = 0}^{2} \rho_{\phi}(i)|_{G_{K}},
$$
cf. \cite{brown2020congruence} Conjecture 7.1.

\subsection{Congruences}

Let
$$
L_1 = \frac{L(2k'+1,\text{Sym}^2 \phi \otimes \chi_K)}{\pi^{2k'+3} \Omega_{\phi}^{+} \Omega_{\phi}^{-}} \in \overline{\Q}, \qquad 
L_2 = \frac{L(2k'+2,\text{Sym}^2 \phi)}{\pi^{2k'+5} \Omega_{\phi}^{+} \Omega_{\phi}^{-}} \in \overline{\Q},
$$
where $L(s,\text{Sym}^2 \phi \otimes \psi))$ is the symmetric square $L$-function for a Dirichlet character $\psi$ (cf. \cite[p.~18]{brown2020congruence}) and $\Omega_{\phi}^{\pm} \in \mathbf{C}^{\times}$ are the integral periods associated with $\phi$ (cf. \cite[p.~24]{brown2020congruence}, \cite{vatsal1999canonical}). Let $v_i = \text{val}_{\varpi} (\# \mO/L_i)$ for $i = 1,2$. The following Lemma states that if the conditions on \(p, \phi, K\) above are satisfied, then the value $v_1, v_2$ controls the existence of an $f'$ such that $f' \equiv I_{\phi} (\mmod \varpi)$.

\begin{lem}{\label{Existance of congruence}}
Suppose $p \nmid 2 D_K, p > 2k'+2$, $\phi$ is ordinary at $p$, and $\overline{\rho}_{\phi}|_{G_K}: G_{K} \to \GL_2(\overline{\F})$ is absolutely irreducible as above. We also assume $\val_{\varpi}(\eta_{\phi}) \ge 0$ and the $L$-value conditions related to $\phi$ such that $\val_{\varpi}(\mathcal{U}) = 0$ (see Theorem 6.1 in \cite{brown2020congruence} for the definitions of $\mathcal{U}$ and $\eta_{\phi}$) and the conditions on $K$ there are satisfied. 

Then $v_1 + v_2 > 0$ translates to $\val_{\varpi}(\mathcal{V}) > 0$ and $\val_{\varpi}(\eta_{\phi} \mathcal{V}) > 0$ in there and there exists $f' \in \mathcal{S}_{3,2k'+2}(\hat{\Z})$ orthogonal to the space spanned by all the Ikeda lifts such that $f' \equiv I_{\phi} (\mmod \varpi^{v_1 + v_2})$ (in the sense of Definition 2.4 in \cite{brown2020congruence}). 
\end{lem}
\begin{proof}
The conditions on \(p, \phi, K\) and $v_1, v_2$ guarantee that the conditions in Theorem 6.1 in \cite{brown2020congruence} are satisfied. The conditions on $K$, including conditions on a Hecke character of $K$ and on $L$-values related to $\chi_K$, were specific to the argument of Theorem 6.1. More details can be found in Section 9 in \cite{brown2020congruence} where the authors verified these conditions by providing examples of congruences.

Then $v_1 + v_2 = \val_{\varpi}(\mathcal{V})$ as in Therorem 6.1 in \cite{brown2011cuspidality} and we have $f' \equiv I_{\phi} (\mmod \varpi^{v_1 + v_2})$.
\end{proof}

\begin{rmk}{\label{6.2}}
We expect according to the $\lambda$-part of the Bloch--Kato conjecture:
$$
v_1 = \text{val}_{\varpi}(\# \mO/L^{\text{alg}}(\text{ad}^0 \rho_{\phi}(2) \otimes \chi_K,0)) = \text{val}_{\varpi}(\# H^1_f(\Q,(\text{ad}^0 \rho_{\phi}(-1) \otimes \chi_K) \otimes E/\mO))
$$
and
$$
v_2 = \text{val}_{\varpi}(\# \mO/L^{\text{alg}}(\text{ad}^0 \rho_{\phi}(3),0)) = \text{val}_{\varpi}(\# H^1_f(\Q,(\text{ad}^0 \rho_{\phi}(-2))\otimes E/\mO)).
$$
\end{rmk}

By Conjecture 7.1 in \cite{brown2020congruence}, which is widely regarded as a known result (see Remark 7.3 in \cite{brown2020congruence}), there is an $p$-adic Galois representation $\rho_{f'}: G_{K} \to \GL_6(E)$ attached to $f'$, and $\rho_{f'}$ is unramified away from $p$ and Fontaine--Laffaille at $p$. Assume that it is also irreducible.

\subsection{$R = \mathbf{T}$ Theorem}

Let $\Sigma = \{v: v \mid p \text{ or } D_K\}$. Let $\overline{\sigma} = \overline{\rho}_{f'}$ and $R$ be the corresponding universal deformation ring defined in Section 3.


Let $\mathcal{S}^{f'}$ denote the subspace spanned by eigenforms $f' \in \mathcal{S}_{3,2k'+2}(\hat{\Z})$ orthogonal to the space spanned by all the Ikeda lifts such that $f' \equiv I_{\phi} (\mmod \varpi^{v_1 + v_2})$ and $\rho_{f'}$ is irreducible. Let $\mathbf{T}^{0}$ be the $\mO$-subalgebra of $\text{End}_{\mO}(\mathcal{S}_{3,2k'+2}(\hat{\Z}))$ generated by the local Hecke algebras away from $\Sigma$, and $\mathbf{T}^{f'}$ be the image of $\mathbf{T}^{0}$ inside $\text{End}_{\mO}(\mathcal{S}^{f'})$. Let $I_{I_{\phi}}$ be the image of $\Ann(I_{\phi}) \subset \mathbf{T}^{0}$ in $\mathbf{T}^{f'}$.


Applying Theorm \ref{R=T}, if
 \begin{enumerate}
    \item $H^{1}_{\Sigma}(K, \ad \overline{\rho}_{\phi}(-2)|_{G_{K}})  = 0$;
    \item $\dim_{k} H^{1}_{\Sigma}(K, \ad \overline{\rho}_{\phi}(1)|_{G_{K}})= 1$;
    \item $\# H^1_{\Sigma}(K, \ad \rho_{\phi}(-1)|_{G_{K}} \otimes E/\mO) \le \# \mO/\varpi^{m}$,
\end{enumerate}
where $m$ is the largest integer satisfying the above inequality, and that $\# \mathbf{T}^{f'}/I_{I_{\phi}} \ge \# \mO/\varpi^m$, then $R = \mathbf{T}^{f'}$. 

As these conditions are difficult to check by the author's experience, we propose equivalent numerical criteria on certain $L$-values as a result of the $\lambda$-part of Bloch--Kato Conjecture below.

\begin{rmk}{\label{6.3}}
By Lemma \ref{torsion Selmer is Selmer torsion}, we have 
$$
H^{1}_{\Sigma}(K, \ad \overline{\rho}_{\phi}(i)|_{G_{K}}) = H^{1}_{\Sigma}(K, \ad \rho_{\phi}(i)|_{G_{K}} \otimes E/\mO) \otimes k.
$$
Also note $\ad \rho_{\phi}(i)|_{G_{K}} \cong \ad^0 \rho_{\phi}(i)|_{G_{K}} \oplus E$, so we have  
$$
H^1_{\Sigma}(K,\ad \rho_{\phi}(i)|_{G_{K}} \otimes E/\mO) \cong H^1_{\Sigma}(K,\ad^0 \rho_{\phi}(i)|_{G_{K}} \otimes E/\mO) \oplus H^1_{\Sigma}(K,E/\mO).
$$
As $p \nmid \# \text{Cl}_{K}$ and $E/\mO$ is a pro-$p$ group, $H^1_{\Sigma}(K,E/\mO)  = 0$, and $H^1_{\Sigma}(K,\ad \rho_{\phi}(i)|_{G_{K}} \otimes E/\mO) = H^1_{\Sigma}(K,\ad^0 \rho_{\phi}(i)|_{G_{K}} \otimes E/\mO)$.
\end{rmk}

\begin{rmk}{\label{6.5}}
Using Lemma \ref{two selmer groups}, $H^1_{\Sigma}(K,\ad^0 \rho_{\phi}(i)|_{G_{K}} \otimes E/\mO) = H^1_{f}(K,\ad^0 \rho_{\phi}(i)|_{G_{K}} \otimes E/\mO)$. Recall to achieve the equality, we need $P_{v}(V^*,1) \in \mO^{*}$ and $(\ad^0 \rho_{\phi}(i)|_{G_{K}} \otimes E/\mO)^{I_{v}}$ is divisible for all places $v$ of $K$ such that $v \mid \ell$ and $\ell \mid D_K$. $\rho_{\phi}$ is unramified at $v$ so $P_{v}(V^*,1) \in \mO^{*}$. Also $(\ad^0 \rho_{\phi}(i)|_{G_{K}} \otimes E/\mO)^{I_{v}} = \ad^0 \rho_{\phi}(i)|_{G_{K}} \otimes E/\mO \cong (E/\mO)^{3}$ is divisible.
\end{rmk}

By Remarks \ref{6.3}, \ref{6.5}, under the assumption of the $\lambda$-part of Bloch--Kato Conjecture (see Remark \ref{Selmer and L} for more details), if the following conditions on $L$-values hold:
\begin{enumerate}
    \item $\text{val}_{\varpi}(\# \mO/L^{\text{alg}}(\text{ad}^0 \rho_{\phi}(3)|_{G_{K}},0)) = 0$,
    \item $\text{val}_{\varpi}(\# \mO/L^{\text{alg}}(\text{ad}^0 \rho_{\phi}|_{G_{K}},0)) = 1$,
    \item $\text{val}_{\varpi}(\# \mO/L^{\text{alg}}(\text{ad}^0 \rho_{\phi}(2)|_{G_{K}},0)) = 1$,
\end{enumerate}
then we can apply Theorem \ref{R=T} to $R$ to prove $R = \mathbf{T}^{f'}$.

\begin{lem}{\label{I_tr = I^tot}}
Let $I_{\tr}$ be the smallest ideal of $R$ containing $\{\tr \sigma (\Frob_{v})  - \sum\limits_{i = 0}^{2} \tr(\rho_{\phi}(i))
\newline
(\Frob_{v}) \big | v \nmid p D_K\}$. Assuming Assumption \ref{main assumptions}(4), then $I_{\tr} = I^{\tot}$.
\end{lem}
\begin{proof}
The proof is analogous to that of Lemma 2.9 in \cite{berger2015lifting}. 
\end{proof}

\begin{rmk}
The Fontaine-Laffaille at $p$ lift of the Jordan-H\"older factors being unique guaranteed by Assumption \ref{main assumptions}(4) is critical to Lemma \ref{I_tr = I^tot}. Without it, one may have $ I^{\tot} \subsetneq I_{\tr}$ (cf. \cite{wake2020rank}).
\end{rmk}

We summarize and conclude the above discussion in Theorem \ref{Ikeda R=T} below.
\begin{thm}{\label{Ikeda R=T}}
($R=\mathbf{T^{f'}}$)
Keeping the assumptions in Lemma \ref{Existance of congruence}, and assume there is no newform $h \in S_{2k'}(1)$ satisfies $\phi \equiv h (\mmod \varpi)$.

If
\begin{enumerate}
    \item $v_1 = 1$,
    \item $v_2 = 0$,
\end{enumerate}
then there exists $f' \in \mathcal{S}_{3,2k'+2}(\hat{\Z})$ orthogonal to the space spanned by all the Ikeda lifts such that $f' \equiv I_{\phi} (\mmod \varpi)$.

Moreover, if
\begin{enumerate}\setcounter{enumi}{2}
    \item $\text{val}_{\varpi}(\# \mO/L^{\text{alg}}(\text{ad}^0 \rho_{\phi},0)) + \text{val}_{\varpi}(\# \mO/L^{\text{alg}}(\text{ad}^0 \rho_{\phi} \otimes \chi_K,0)) = 1$,
    \item $\text{val}_{\varpi}(\# \mO/L^{\text{alg}}(\text{ad}^0 \rho_{\phi}(2),0)) = 0$,
    \item $\text{val}_{\varpi}(\# \mO/L^{\text{alg}}(\text{ad}^0 \rho_{\phi}(3)  \otimes \chi_K,0))  = 0$,
\end{enumerate}
and the Galois representation $\rho_{f'}: G_{K} \to \GL_6(E)$ attached to $f'$ is irreducible, then the $\lambda$-part of the Bloch--Kato conjecture implies $\# \mathbf{T}^{f'}/I_{I_{\phi}} \ge \# \mO/\varpi$, 
\newline
$\# H^1_{\Sigma}(K, \ad \rho_{\phi}(-1)|_{G_{K}} \otimes E/\mO) \le \# \mO/\varpi$ and $R \cong \mathbf{T}^{f'}$.
\end{thm}
\begin{proof}
The existence of an \(f'\) such that $f' \equiv I_{\phi} (\mmod \varpi)$ and $\# \mathbf{T}^{f'}/I_{I_{\phi}} \ge \# \mO/\varpi$ follows from Lemma \ref{Existance of congruence}.

Regarding the statements related to $R \cong \mathbf{T}^{f'}$, we apply Theorem \ref{R=T}. First, we show $\overline{\sigma}$ satisfies the setup in Theorem \ref{R=T} and Section \ref{general residual set up}. By Conjecture 7.1 in \cite{brown2020congruence}, which is widely regarded as a known result (see Remark 7.3 in \cite{brown2020congruence}), since $p \nmid D_K$ and $p > 3$ by assumption, $\overline{\sigma}$ is Fontaine--Laffaille at $p$. Note that $L^{\text{alg}}(\text{ad}^0 \rho_{\phi}(i)|_{G_{K}},0) = L^{\text{alg}}(\text{ad}^0 \rho_{\phi}(i),0) L^{\text{alg}}(\text{ad}^0 \rho_{\phi}(i) \otimes \chi_K,0)$ for $i \in \Z$ by an easy calculation on Euler factors. Then following Proposition \ref{Lattice}, under the assumptions (2)(5) in the Theorem and the $\lambda$-part of the Bloch--Kato conjecture, we can find a Galois-invariant lattice $T$ inside the representation, such that the reduction of $\rho$ attached to $T$ is 
$\overline{\rho}_{f'}: G_{K} \to \GL_6(k)$:
$$
\overline{\rho}_{f'}(g) = 
\overline{\epsilon}^{2-k'}(g)
\begin{pmatrix}
     \overline{\rho}_{\phi}|_{G_{K}}(g) & a(g) & b(g)\\
     & \overline{\rho}_{\phi}|_{G_{K}}(g) \overline{\epsilon}(g) & c(g) \\
     & & \overline{\rho}_{\phi}|_{G_{K}}(g) \overline{\epsilon}^2(g)
\end{pmatrix},
$$
where
$
\begin{pmatrix}
    \overline{\rho}_{\phi}|_{G_{K}} & a\\
    0 & \overline{\rho}_{\phi}|_{G_{K}} \overline{\epsilon}
\end{pmatrix},
\begin{pmatrix}
    \overline{\rho}_{\phi}|_{G_{K}} \overline{\epsilon} & c\\
    0 & \overline{\rho}_{\phi}|_{G_{K}} \overline{\epsilon}^2
\end{pmatrix}
$
are non-split extensions. And the map $\tau: \mO[\G] \to \mO[\G]$ by $\tau(g) = \epsilon(g) g^{-1} \epsilon^2$ guarantees the existence of the anti-involution.

Then, we check the running Assumptions \ref{main assumptions}. By the $\lambda$-part of the Bloch--Kato conjecture, numerical assumptions (1)--(5) in Theorem \ref{Ikeda R=T} ensure (1)--(3) in Assumptions \ref{main assumptions}. For (4) in Assumptions \ref{main assumptions}, there is no newform $h \in S_{2k'}(1)$ such that $\phi \equiv h (\mmod \varpi)$, so $R_{i,\FL} = \mO$ for each $i$ (see Remark \ref{R_i = O}).

Following a similar argument as in the proof of Proposition 7.13 in \cite{berger2013deformation}, we get that $R$ is generated by traces and there is a $\mO$-algebra surjection: $\varphi: R \twoheadrightarrow \mathbf{T}^{f'}$. By a similar argument as in Lemma 3.7 in \cite{berger2015lifting}, we can show the inverse image of $I_{I_{\phi}}$ is $I_{\tr}$, and we get that $I_{I_{\phi}} = \varphi(I^{\tot})$ by Lemma \ref{I_tr = I^tot}. 

The existence of \(f'\) such that $f' \equiv I_{\phi} (\mmod \varpi^{v_1 + v_2})$ gives that $\# \mathbf{T}^{f'}/I_{I_{\phi}} \ge \# \mO/\varpi^{v_1 + v_2}$ \cite{berger2013higher}. 
Along with the $\lambda$-part of the Bloch--Kato conjecture that 
\newline
says $\# H^1_{\Sigma}(K, \ad \rho_{\phi}(-1)|_{G_{K}} \otimes E/\mO) \le \# \mO/\varpi^{v_1^{'}}$, where $v_1^{'} = \text{val}_{\varpi}(\# \mO/L^{\text{alg}}(\text{ad}^0 \rho_{\phi}(2),0)) + v_1$ by the $L$-value discussion above, we conclude $R \cong \mathbf{T}^{f'}$ from Theorem \ref{R=T}.

\end{proof}

\begin{rmk}
Theorem \ref{Ikeda R=T} is an application of Theorem \ref{R=T}, where \(R\) is not guaranteed to be a DVR. However, in Theorem \ref{Ikeda R=T}, $R$ is still a DVR. This is because we are using the $L$-value $v_1^{'}$ to simultaneously bound $\# H^1_{\Sigma}(K, \ad \rho_{\phi}(-1)|_{G_{K}} \otimes E/\mO)$ and ensure $\dim_{k} H^{1}_{\Sigma}(K, \ad \overline{\rho}_{\phi}(=1)|_{G_{K}}) = 1$. 

However, in Theorem \ref{R=T}, if one can directly determine these values separately, it is possible that $\# H^1_{\Sigma}(F,\Hom(\Tilde{\rho_2},\Tilde{\rho_1}) \otimes_{E} E/\mO) > \#\mO/\varpi$ and $\dim_{k} H^1_{\Sigma}(F,\Hom(\rho_2,\rho_1) = 1$ and $R$ is not a DVR. See Remark \ref{Selma not cyclic} for an example.
\end{rmk}

\printbibliography

\end{document}